\documentclass[12pt,reqno,draft]{amsart} 
\usepackage{amssymb,amscd,url}

\begin{document}

\newif\ifArXivVersion
\ArXivVersiontrue

\ifArXivVersion
\def\EndNote#1{\footnote{For further details, see Appendix \ref{#1}.}}
\else
\def\EndNote#1{}
\fi

\allowdisplaybreaks


\title[Dynamical Degree and Canonical Height]
{Dynamical Degree, Arithmetic Entropy, and Canonical Heights for 
Dominant Rational Self-Maps of Projective Space}
\date{\today}
\author[Joseph H. Silverman]{Joseph H. Silverman}
\email{jhs@math.brown.edu}
\address{Mathematics Department, Box 1917
         Brown University, Providence, RI 02912 USA}
\subjclass[2010]{Primary: 37P30; Secondary:  11G50, 37F10, 37P15}
\keywords{dynamical degree, arithmetic entropy, canonical height, rational map}
\thanks{The author's research supported by DMS-0854755.}



\hyphenation{ca-non-i-cal semi-abel-ian}


\newtheorem{theorem}{Theorem}
\newtheorem{lemma}[theorem]{Lemma}
\newtheorem{conjecture}[theorem]{Conjecture}
\newtheorem{proposition}[theorem]{Proposition}
\newtheorem{corollary}[theorem]{Corollary}
\newtheorem*{claim}{Claim}

\theoremstyle{definition}
\newtheorem*{definition}{Definition}
\newtheorem{example}[theorem]{Example}
\newtheorem{remark}[theorem]{Remark}
\newtheorem{question}[theorem]{Question}

\theoremstyle{remark}
\newtheorem*{acknowledgement}{Acknowledgements}
\newtheorem*{addendum}{Addendum}


\newenvironment{notation}[0]{%
  \begin{list}%
    {}%
    {\setlength{\itemindent}{0pt}
     \setlength{\labelwidth}{4\parindent}
     \setlength{\labelsep}{\parindent}
     \setlength{\leftmargin}{5\parindent}
     \setlength{\itemsep}{0pt}
     }%
   }%
  {\end{list}}

\newenvironment{parts}[0]{%
  \begin{list}{}%
    {\setlength{\itemindent}{0pt}
     \setlength{\labelwidth}{1.5\parindent}
     \setlength{\labelsep}{.5\parindent}
     \setlength{\leftmargin}{2\parindent}
     \setlength{\itemsep}{0pt}
     }%
   }%
  {\end{list}}
\newcommand{\Part}[1]{\item[\upshape#1]}

\def\Case#1#2{%
\paragraph{\textbf{\boldmath Case #1: #2.}}\hfil\break\ignorespaces}

\renewcommand{\a}{\alpha}
\renewcommand{\b}{\beta}
\newcommand{\g}{\gamma}
\renewcommand{\d}{\delta}
\newcommand{\e}{\epsilon}
\newcommand{\f}{\varphi}
\newcommand{\bfphi}{{\boldsymbol{\f}}}
\renewcommand{\l}{\lambda}
\renewcommand{\k}{\kappa}
\newcommand{\lhat}{\hat\lambda}
\newcommand{\m}{\mu}
\newcommand{\bfmu}{{\boldsymbol{\mu}}}
\renewcommand{\o}{\omega}
\renewcommand{\r}{\rho}
\newcommand{\s}{\sigma}
\newcommand{\sbar}{{\bar\sigma}}
\renewcommand{\t}{\tau}
\newcommand{\z}{\zeta}
\newcommand{\bfzeta}{{\boldsymbol{\zeta}}}

\newcommand{\D}{\Delta}
\newcommand{\G}{\Gamma}
\newcommand{\F}{\Phi}
\renewcommand{\L}{\Lambda}

\newcommand{\ga}{{\mathfrak{a}}}
\newcommand{\gb}{{\mathfrak{b}}}
\newcommand{\gn}{{\mathfrak{n}}}
\newcommand{\gp}{{\mathfrak{p}}}
\newcommand{\gP}{{\mathfrak{P}}}
\newcommand{\gq}{{\mathfrak{q}}}

\newcommand{\Abar}{{\bar A}}
\newcommand{\Ebar}{{\bar E}}
\newcommand{\kbar}{{\bar k}}
\newcommand{\Kbar}{{\bar K}}
\newcommand{\Pbar}{{\bar P}}
\newcommand{\Sbar}{{\bar S}}
\newcommand{\Tbar}{{\bar T}}
\newcommand{\gbar}{{\bar\gamma}}
\newcommand{\lbar}{{\bar\lambda}}
\newcommand{\ybar}{{\bar y}}
\newcommand{\phibar}{{\bar\f}}

\newcommand{\Acal}{{\mathcal A}}
\newcommand{\Bcal}{{\mathcal B}}
\newcommand{\Ccal}{{\mathcal C}}
\newcommand{\Dcal}{{\mathcal D}}
\newcommand{\Ecal}{{\mathcal E}}
\newcommand{\Fcal}{{\mathcal F}}
\newcommand{\Gcal}{{\mathcal G}}
\newcommand{\Hcal}{{\mathcal H}}
\newcommand{\Ical}{{\mathcal I}}
\newcommand{\Jcal}{{\mathcal J}}
\newcommand{\Kcal}{{\mathcal K}}
\newcommand{\Lcal}{{\mathcal L}}
\newcommand{\Mcal}{{\mathcal M}}
\newcommand{\Ncal}{{\mathcal N}}
\newcommand{\Ocal}{{\mathcal O}}
\newcommand{\Pcal}{{\mathcal P}}
\newcommand{\Qcal}{{\mathcal Q}}
\newcommand{\Rcal}{{\mathcal R}}
\newcommand{\Scal}{{\mathcal S}}
\newcommand{\Tcal}{{\mathcal T}}
\newcommand{\Ucal}{{\mathcal U}}
\newcommand{\Vcal}{{\mathcal V}}
\newcommand{\Wcal}{{\mathcal W}}
\newcommand{\Xcal}{{\mathcal X}}
\newcommand{\Ycal}{{\mathcal Y}}
\newcommand{\Zcal}{{\mathcal Z}}

\renewcommand{\AA}{\mathbb{A}}
\newcommand{\BB}{\mathbb{B}}
\newcommand{\CC}{\mathbb{C}}
\newcommand{\FF}{\mathbb{F}}
\newcommand{\GG}{\mathbb{G}}
\newcommand{\NN}{\mathbb{N}}
\newcommand{\PP}{\mathbb{P}}
\newcommand{\QQ}{\mathbb{Q}}
\newcommand{\RR}{\mathbb{R}}
\newcommand{\ZZ}{\mathbb{Z}}

\newcommand{\bfa}{{\mathbf a}}
\newcommand{\bfb}{{\mathbf b}}
\newcommand{\bfc}{{\mathbf c}}
\newcommand{\bfd}{{\mathbf d}}
\newcommand{\bfe}{{\mathbf e}}
\newcommand{\bff}{{\mathbf f}}
\newcommand{\bfg}{{\mathbf g}}
\newcommand{\bfp}{{\mathbf p}}
\newcommand{\bfr}{{\mathbf r}}
\newcommand{\bfs}{{\mathbf s}}
\newcommand{\bft}{{\mathbf t}}
\newcommand{\bfu}{{\mathbf u}}
\newcommand{\bfv}{{\mathbf v}}
\newcommand{\bfw}{{\mathbf w}}
\newcommand{\bfx}{{\mathbf x}}
\newcommand{\bfy}{{\mathbf y}}
\newcommand{\bfz}{{\mathbf z}}
\newcommand{\bfA}{{\mathbf A}}
\newcommand{\bfF}{{\mathbf F}}
\newcommand{\bfB}{{\mathbf B}}
\newcommand{\bfD}{{\mathbf D}}
\newcommand{\bfG}{{\mathbf G}}
\newcommand{\bfI}{{\mathbf I}}
\newcommand{\bfM}{{\mathbf M}}
\newcommand{\bfzero}{{\boldsymbol{0}}}

\newcommand{\Aut}{\operatorname{Aut}}
\newcommand{\codim}{\operatorname{codim}}
\newcommand{\diag}{\operatorname{diag}}
\newcommand{\Disc}{\operatorname{Disc}}
\newcommand{\Div}{\operatorname{Div}}
\renewcommand{\div}{{\textup{div}}}
\newcommand{\Dom}{\operatorname{Dom}}
\newcommand{\End}{\operatorname{End}}
\newcommand{\Ext}{\operatorname{Ext}}
\newcommand{\Fbar}{{\bar{F}}}
\newcommand{\Gal}{\operatorname{Gal}}
\newcommand{\GL}{\operatorname{GL}}
\newcommand{\Hom}{\operatorname{Hom}}
\newcommand{\Index}{\operatorname{Index}}
\newcommand{\Image}{\operatorname{Image}}
\newcommand{\hhat}{{\hat h}}
\newcommand{\Ker}{{\operatorname{ker}}}
\newcommand{\Lift}{\operatorname{Lift}}
\newcommand{\Mat}{\operatorname{Mat}}
\newcommand{\maxplus}{\operatorname{\textup{max}^{\scriptscriptstyle+}}}
\newcommand{\MOD}[1]{~(\textup{mod}~#1)}
\newcommand{\Mor}{\operatorname{Mor}}
\newcommand{\Moduli}{\mathcal{M}}
\newcommand{\Norm}{{\operatorname{\mathsf{N}}}}
\newcommand{\notdivide}{\nmid}
\newcommand{\normalsubgroup}{\triangleleft}
\newcommand{\NS}{\operatorname{NS}}
\newcommand{\onto}{\twoheadrightarrow}
\newcommand{\ord}{\operatorname{ord}}
\newcommand{\Orbit}{\mathcal{O}}
\newcommand{\Per}{\operatorname{Per}}
\newcommand{\Perp}{\operatorname{Perp}}
\newcommand{\PrePer}{\operatorname{PrePer}}
\newcommand{\PGL}{\operatorname{PGL}}
\newcommand{\Pic}{\operatorname{Pic}}
\newcommand{\Prob}{\operatorname{Prob}}
\newcommand{\Qbar}{{\bar{\QQ}}}
\newcommand{\rank}{\operatorname{rank}}
\newcommand{\Rat}{\operatorname{Rat}}
\newcommand{\rbar}{{\overline{r}}}
\newcommand{\Resultant}{\operatorname{Res}}
\renewcommand{\setminus}{\smallsetminus}
\newcommand{\sgn}{\operatorname{sgn}} 
\newcommand{\SL}{\operatorname{SL}}
\newcommand{\Span}{\operatorname{Span}}
\newcommand{\Spec}{\operatorname{Spec}}
\newcommand{\Support}{\operatorname{Supp}}
\newcommand{\tors}{{\textup{tors}}}
\newcommand{\Trace}{\operatorname{Trace}}
\newcommand{\tr}{{\textup{tr}}} 
\newcommand{\UHP}{{\mathfrak{h}}}    
\newcommand{\<}{\langle}
\renewcommand{\>}{\rangle}

\newcommand{\ds}{\displaystyle}
\newcommand{\longhookrightarrow}{\lhook\joinrel\longrightarrow}
\newcommand{\longonto}{\relbar\joinrel\twoheadrightarrow}


\begin{abstract}
Let $\f:\PP^N\dashrightarrow\PP^N$ be a dominant rational map. The
dynamical degree of~$\f$ is the quantity $\d_\f=\lim(\deg\f^n)^{1/n}$.
When~$\f$ is defined over~$\Qbar$, we define the arithmetic degree of
a point~$P\in\PP^N(\Qbar)$ to be~$\a_\f(P)=\limsup
h\bigl(\f^n(P)\bigr)^{1/n}$ and the canonical height of~$P$ to
be~$\hhat_\f(P)=\limsup \d_\f^{-n}n^{-\ell_\f}h\bigl(\f^n(P)\bigr)$
for an appropriately chosen~$\ell_\f$. In this article we begin by
proving some elementary relations and making some deep conjectures
relating~$\d_\f$,~$\a_\f(P)$, $\hhat_\f(P)$, and the Zariski density of
the orbit~$\Orbit_\f(P)$ of~$P$. We then prove our conjectures for
monomial maps.
\end{abstract}


\maketitle

\ifArXivVersion
\setcounter{tocdepth}{1} 
\tableofcontents 
\fi

\section{Introduction}
Let $\f:\PP^N\dashrightarrow\PP^N$ be a dominant rational map, that
is, a map given by homogeneous polynomials~$\f_0,\ldots,\f_N$ of the
same degree having no common nontrivial factors. The map~$\f$ is
called \emph{algebraically stable}~\cite{MR1369137} if
\[
  \deg (\f^n) = (\deg\f)^n\quad\text{for all $n\ge1$.}
\]
Examples of algebraically stable maps include morphisms and regular
affine automorphisms.
\par
In this paper we are principally concerned with the geometry and
arithmetic of maps that are not algebraically stable. The (first)
\emph{dynamical degree of~$\f$} is defined by
\[
  \d_\f = \lim_{n\to\infty} \bigl(\deg(\f^n)\bigr)^{1/n},
\]
and $\log\d_\f$ is sometimes called the \emph{algebraic entropy
of~$\f$}; see~\cite{MR1704282}.  The extent to which~$\d_\f$ differs
from~$\deg(\f)$ is a rough measure of the failure of~$\f$ to be
algebraically stable.  Dynamical degrees were initially studied by
Russakovskii and Shiffman~\cite{MR1488341} and
Arnol$'$d~\cite{MR1139553} in the 1990s, and they have since attracted
considerable attention; see for example~\cite{MR2111418,
  MR2449533, MR2428100, MR1877828, MR1867314, MR2180409, MR2339287,
  arxiv1011.2854, MR2358970, arxiv1007.0253, arxiv1010.6285, 
  MR1836434, MR2273670,arxiv1106.1825}.  
Bellon and Viallet~\cite{MR1704282} conjectured that~$\d_\f$ is an
algebraic integer, while Hasselblatt and Propp~\cite{MR2358970} (see
also~\cite{MR2449533}) proved that the sequence~$\deg(\f^n)$ may be
quite irregular in the sense that the power series $\sum_{n\ge0}
\deg(\f^n)T^n$ need not be a rational function.
\par
The primary objectives of this paper are to study an arithmetic analogue
of the dynamical degree and to define an associated canonical height
function for dominant rational maps.  In this introduction we
make a number of conjectures, which we will prove
for monomial maps.
\par
So we now assume that~$\f$ is defined over~$\Qbar$, and we consider the
iterates of~$\f$ applied to points in~$\PP^N(\Qbar)$.  Let
\[
  h:\PP^N(\Qbar)\to[0,\infty)
\]
denote the usual Weil height; see,
e.g.,~\cite{bombierigubler,hindrysilverman:diophantinegeometry,lang:diophantinegeometry,MR2316407,MR2514094}
for definitions and basic properties of~$h$. An elementary triangle
inequality estimate shows that~$h\bigl(\f^n(P)\bigr)\ll(\deg\f)^n$.
For points~$P\in\PP^N(\Qbar)$ whose orbit~$\Orbit_\f(P)$ is disjoint
from the indeterminacy locus~$Z(\f)$ of~$\f$, we define the
\emph{arithmetic degree of~$\f$ at~$P$} to be the quantity
\[
  \a_\f(P) = \limsup_{n\to\infty} h\bigl(\f^n(P)\bigr)^{1/n}.
\]
\par
We note that since~$h(P)$ is, roughly, the information-theoretic
content of~$P$, it is reasonable to say that~$\log\a_\f(P)$
measures the \emph{arithmetic entropy} of the orbit~$\Orbit_\f(P)$.
It is not hard to show (Proposition~\ref{proposition:hfnPledfnhP})
that
\begin{equation}
  \label{eqn:afPledf}
  \a_\f(P) \le \d_\f.
\end{equation}
The fact that~\eqref{eqn:afPledf} may be a strict inequality
reflects that fact that some orbits capture only a part of the
complexity of the map~$\f$. Our first conjecture
describes a sufficient condition for equality.
\par
We set the notation
\[
  \PP^N(\Qbar)_\f = \bigl\{
    \text{$P\in\PP^N(\Qbar)$ such that $\Orbit_\f(P)\cap Z(\f)=\emptyset$}
  \bigr\}.
\]
We remark that $\PP^N(\Qbar)_\f$ is Zariski dense in~$\PP^N$,
although the proof is not easy; see for example~\cite{MR2784670}.

\begin{conjecture}
\label{conjecture:setofafPs}
Let $\f:\PP^N\dashrightarrow\PP^N$ be a dominant rational map defined
over~$\Qbar$. 
\begin{parts}
\Part{(a)}
The set
\[
  \bigl\{\a_\f(P) :  P\in\PP^N(\Qbar)_\f \bigr\}
\]
is a finite set of algebraic integers. 
\Part{(b)}
Let $P\in\PP^N(\Qbar)_\f$ be a point such that $\Orbit_\f(P)$ is
Zariski dense in~$\PP^N$. Then $\a_\f(P)=\d_\f$.
\end{parts}
\end{conjecture}

In Section~\ref{section:monomialmapht0} we prove
Conjecture~\ref{conjecture:setofafPs} for monomial maps on~$\PP^N$.
\par 
As noted earlier, a major objective of this paper is to define and study
canonical heights for general dominant rational maps.  We
recall~\cite{callsilv:htonvariety,MR2316407} that
if~$\f:\PP^N\to\PP^N$ is a \emph{morphism} of degree~$d\ge2$, then the
canonical height associated to~$\f$ is the function
\[
  \hhat_\f:\PP^N(\Qbar)\longrightarrow [0,\infty),\qquad
  \hhat_\f(P) = \lim_{n\to\infty} \frac{1}{d^n}h\bigl(\f^n(P)\bigr).
\]
The canonical height for a morphism is characterized by the
properties
\[
  \hhat_\f(P)=h(P)+O(1)\quad\text{and}\quad
  \hhat_\f\bigl(\f(P)\bigr)=d\hhat_\f(P),
\]
from which one easily deduces that
\[
  P\in\PrePer(\f)\quad\Longleftrightarrow\quad \hhat_\f(P)=0.
\]
\par
For general dominant rational maps we have $\deg(\f^n) \approx
\d_\f^n$, so it is natural to look at
\[
  \frac{1}{\d_\f^n}h\bigl(\f^n(P)\bigr),
\]
but the approximation $\deg(\f^n) \approx \d_\f^n$ is insufficiently
precise. For example, the map $\f(x,y)=(x^d y,y^d)$
satisfies~$\deg(\f^n)=d^n+nd^{n-1}$, so~$\deg(\f^n)$ grows faster
than~$\d_\f^n=d^n$. This leads us to make the following conjecture
(cf.\ \cite{MR2358970}), which will provide the required correction
factor.

\begin{conjecture}
\label{conjecture:dyndegtimespoly}
Let~$\f:\PP^N\dashrightarrow\PP^N$ be a dominant rational map.
Then the infimum
\[
  \ell_\f = \inf\left\{ 
    \ell\ge0 : \sup_{n\ge1} \frac{\deg(\f^n)}{n^{\ell}\,\d_\f^{n}} < \infty \right\}
\]
exists and is an integer satisfying $0\le\ell_\f\le N$.
\end{conjecture}

We note that Conjecture~\ref{conjecture:dyndegtimespoly} is really
three conjectures, first that $\d_\f^{-n}\deg(\f^n)$ grows at most
polynomially in~$n$, second that the growth rate is
essentially~$n^\ell$ for an integer~$\ell$, and third that~$\ell$ is
between~$0$ and~$N$. As noted by the referee, a bold person might even
conjecture that $\deg(\f^n)\asymp n^{\ell_\f}\d_\f^n$, which would
preclude for example the appearence of powers of~$\log n$ in the
growth rate.  See Section~\ref{subsection:ellconj} for a further
discussion of Conjecture~\ref{conjecture:dyndegtimespoly}.

\begin{definition}
\label{definition:canht}
With notation as above, the \emph{canonical height
  of~$P\in\PP^N(\Qbar)_\f$ with respect to~$\f$} is
\[
  \hhat_\f(P) = \limsup_{n\to\infty} \frac{1}{n^{\ell_\f}\d_\f^n}h\bigl(\f^n(P)\bigr).
\]
\end{definition}

We note that the limsup is necessary, since it is easy to construct
examples for which the limit diverges by oscillation; see
Example~\ref{example:hhatnotlimit}. Also, it is easy to check
(Proposition~\ref{proposition:canhtprops}) that
\[
  \hhat_\f\bigl(\f(P)\bigr)=\d_\f\hhat_\f(P).
\]
If~$\d_\f>1$, we suspect that~$\hhat_\f(P)$ is finite, and we prove
that this holds for monomial maps
(Proposition~\ref{proposition:adjhtformonmaps}).  However, if
$\d_\f=1$, then it is possible to have~$\ell_\f\ge1$ and
$\hhat_\f(P)=\infty$, as we show in Example~\ref{example:infiniteht}.
\par
It is not hard to prove that
\begin{equation}
  \label{eqn:hfPgt0impafPeqdf}
  \hhat_\f(P)>0\quad\Longrightarrow\quad\a_\f(P)=\d_\f;
\end{equation}
see Proposition~\ref{proposition:canhtprops}(d). The converse
to~\eqref{eqn:hfPgt0impafPeqdf} is not true in general; see the
discussion before the statement of
Corollary~\ref{corollary:hgt0iffaeqd}. It would be very interesting to
find general geometric conditions on~$\f$ that imply the converse
of~\eqref{eqn:hfPgt0impafPeqdf}. We  prove in
Corollary~\ref{corollary:hgt0iffaeqd} that the converse holds for
monomial maps associated to diagonalizable matrices.

A fundamental property of the canonical height for morphisms is that
height zero characterizes points with finite orbit.  (N.B. We always
work over~$\Qbar$. The situation over function fields is subtler; see
for example~\cite{arxiv0601046,arxiv0510444}.)  For any dominant
rational map~$\f$ with $\d_\f>1$ or $\ell_\f>0$, we clearly have
\[
  P\in\PrePer(\f)\quad\Longrightarrow\quad \hhat_\f(P)=0,
\]
but the converse is not true in general, since there may be
subvarieties on which~$\f$ acts via lower degree.  This leads to the
following conjecture.

\begin{conjecture}
\label{conjecture:hzeroeqPPrePer}
Let~$\f:\PP^N\dashrightarrow\PP^N$ be a dominant rational map defined
over~$\Qbar$ with dynamical degree $\d_\f>1$, and
let~$P\in\PP^N(\Qbar)_\f$ be a point
whose orbit~$\Orbit_\f(P)$ is Zariski dense
in~$\PP^N$.  Then $\hhat_\f(P)>0$.
\end{conjecture}

We observe that Conjecture~\ref{conjecture:hzeroeqPPrePer} and the
elementary implication~\eqref{eqn:hfPgt0impafPeqdf} imply
Conjecture~\ref{conjecture:setofafPs}(b).

The main theorem in this paper (Theorem~\ref{theorem:linrelsonlogP})
gives a geometric description of the set of points satisfying
$\hhat_\f(P)=0$ for monomial maps~$\f$. Immediate corollaries include
proofs of Conjectures~\ref{conjecture:setofafPs}
and~\ref{conjecture:hzeroeqPPrePer} for monomial maps.  We also note
that a strong form of Conjecture~\ref{conjecture:dyndegtimespoly} is
true for monomial maps; this was proven independently by
Lin~\cite{arxiv1007.0253} and Jonsson and
Wulcan~\cite{arXiv:1001.3938}.

We recall that a \emph{monomial map} is an endomorphism of the
torus~$\GG_m^N$, i.e., a map
\[
  \f_A : \GG_m^N\longrightarrow\GG_m^N
\]
of the form
\[
  \f_A(X_1,\ldots,X_N)={}
   \left(
    X_1^{a_{11}}X_2^{a_{12}}\cdots X_N^{a_{1N}},\ldots,
    X_1^{a_{N1}}X_2^{a_{N2}}\cdots X_N^{a_{NN}} \right),
\]
where~$A=(a_{ij})$ is an $N$-by-$N$ matrix with integer coefficients.
The associated rational map~$\f_A:\PP^N\dashrightarrow\PP^N$ is
dominant if $\det(A)\ne0$.  Hasselblatt and Propp~\cite{MR2358970}
have shown that the dynamical degree of~$\f_A$ is equal to the
spectral radius of~$A$, i.e., the magnitude of the largest eigenvalue
of~$A$.

The following is a special case of our main theorem and its
corollaries; see Section~\ref{section:monomialmapht0} for details.

\begin{theorem}
\label{theorem:conj134formonmaps}
Conjectures~$\ref{conjecture:setofafPs}$
and~$\ref{conjecture:hzeroeqPPrePer}$ are true for monomial maps. More
precisely, let~$\f_A$ be a monomial map with $\d_{\f_A}>1$. 
\begin{parts}
\Part{(a)}
The arithmetic degrees of~$\f_A$ satisfy
\[
  \bigl\{\a_{\f_A}(P):P\in\GG_m^N(\Qbar)\bigr\}
  \subset \bigl\{\text{eigenvalues of $A$}\bigr\}.
\]
In particular,~$\a_{\f_A}(P)$ is an algebraic integer for all
$P\in\GG_m^N(\Qbar)$.
\Part{(b)}
Let~$P\in\GG_m^N(\Qbar)$ be a point with
$\hhat_{\f_A}(P)=0$. Then~$\Orbit_\f(P)$ is contained in a proper
\textup(possibly disconnected\textup) algebraic subgroup of $\GG_m^N$.
\Part{(c)}
If the matrix $A$ is diagonalizable over~$\CC$, then
\[
  \hhat_\f(P)=0\quad\Longleftrightarrow\quad\a_\f(P)<\d_\f.
\]
\Part{(d)}
If the characteristic polynomial of
the matrix $A$ is irreducible over $\QQ$, then
\[
  \hhat_\f(P)=0\quad\Longleftrightarrow\quad 
  \text{$\Orbit_{\f_A}(P)$ is finite}.
\]
\end{parts}
\end{theorem}

Theorem~\ref{theorem:conj134formonmaps} is proven in
Section~\ref{section:proofofcorollaries} as a series of corollaries to
Theorem~\ref{theorem:linrelsonlogP}, which is our main result.  The
proof of Theorem~\ref{theorem:linrelsonlogP} uses a compactness
argument, the product formula, Baker's theorem on
linear-forms-in-logarithms, and a lot of linear algebra.  In
particular, Baker's theorem is needed to show
the~$\Gal(\Qbar/\QQ)$-invariance of the set of~$\Qbar$-linear
relations on a set of log absolute values
\[
  \log\|x_1\|_v,\ldots,\log\|x_N\|_v,
\]
where the~$x_i$ are in~$\Qbar$ and~$v\in M_\Qbar$ is an (archimedean)
absolute value on~$\Qbar$.

We conclude the paper with two additional results.  In
Section~\ref{section:regaffaut} we use Kawagu\-chi's theory of
canonical heights for regular affine automorphisms to prove our
conjectures for maps of this type, and in
Section~\ref{section:generalizations} we generalize
Conjecture~\ref{conjecture:setofafPs} to dominant rational self-maps
of arbitrary (nonsingular) varieties and prove that it is true for
automorphisms of certain~K3 surfaces.

\begin{addendum}
While this paper was under review, a number of authors have written
papers that grew out of the questions raised and results proven in
this paper. We mention in particular a paper of Jonsson and
Wulcan~\cite{arxiv1202.0203} in which they prove much of
Conjectures~\ref{conjecture:setofafPs}
and~\ref{conjecture:hzeroeqPPrePer} for polynomial
morphisms~$\f:\AA^2\to\AA^2$ of small topological degree, and a paper
of Kawaguchi and the author~\cite{kawsilv:arithdegledyndeg} in which
it is shown that $\a_\f(P)\le\d_\f$ holds for dominant rational
self-maps of (normal) varieties.
\end{addendum}

\ifArXivVersion
\noindent\emph{Remark.}\enspace
In the ArXiv version of this article, for the convenience of the
reader we have an included an appendix giving further details about
various elemenatry remarks and assertions. The appendix will not
appear in the published version.  
\fi

\begin{acknowledgement}
The author thanks Charles Favre, Mattias Jonsson, Shu Kawaguchi,
Jan-Li Lin, James Propp, Juan Rivera-Letelier, and Tom Ward for their
helpful comments on the initial draft. The author also
thanks Mattias Jonsson for pointing out that the converse
to~\eqref{eqn:hfPgt0impafPeqdf} does not hold in general, Jan-Li Lin
for showing the author the short proof of Lemma~\ref{lemma:fAPeqfBP}
(which improved the original proof that worked only over~$\Qbar$), and
the referee for his careful reading of the manuscript and his/her many
suggestions, including especially a simplification and generalization
of the proof of Theorem~\ref{theorem:linrelsonlogP} that eliminated
the assumption that the matrix~$A$ be diagonalizable.
\end{acknowledgement}

\section{Relation to earlier work}
\label{section:earlierwork}

\subsection{Growth rate of $\protect{\boldsymbol{\deg(\f^n)}}$}
\label{subsection:ellconj}
Conjecture~\ref{conjecture:dyndegtimespoly} is related to questions
raised by Hasselblatt and Propp~\cite{MR2358970}.  In particular, they
ask~\cite[Question~9.5]{MR2358970} if the degree sequence~$\deg(\f^n)$
can be simultaneously subexponential and
superpolynomial. Conjecture~\ref{conjecture:dyndegtimespoly} says that
this cannot happen. They further ask~\cite[Question~9.6]{MR2358970}
if, whenever~$\deg(\f^n)$ is bounded by a power of~$n$, must it grow
essentially like~$n^\ell$ for a non-negative integer~$\ell$.
Conjecture~\ref{conjecture:dyndegtimespoly} says that this is true, so
for example a growth rate of order~$\sqrt{n}$ should not be
possible.
\par
The classification results of Diller and Favre~\cite{MR1867314} can be
used to show that Conjecture~\ref{conjecture:dyndegtimespoly} is true
for birational maps of~$\PP^2$ having~$\d_\f=1$. See also~\cite{MR2428100}
for families of birational maps on~$\PP^2$ having~$\d_\f=1$
and~$\ell_\f=2$, which shows that~$\ell_\f$ may be as large as the
dimension.
Lin~\cite{arxiv1007.0253} and Jonsson and
Wulcan~\cite{arXiv:1001.3938} have shown that a strong form of
Conjecture~\ref{conjecture:dyndegtimespoly} holds for monomial maps;
see Theorem~\ref{theorem:ellinvariantmonomap}.  See
also~\cite{arxiv0608267,MR2339287,arxiv0711.2770} for a proof that
Conjecture~\ref{conjecture:dyndegtimespoly} holds for certain
rational maps of~$\PP^2$, including in particular all polynomial maps
of~$\AA^2$.
\par
In general, it is very difficult to compute, or even to
estimate, the value of the dynamical degree of a rational map in
dimension greater than~$2$, since even on a computer one generally
cannot compute the map~$\f^n$ for moderate values
of~$n$. See~\cite{MR2220002,MR2111418,arxiv0512507} for some
discussion of these issues and for the computation of~$\d_\f$ for
certain higher-di\-men\-sional maps.

\subsection{Canonical heights for regular affine automorphisms}
\label{subsection:regaffauts}
\leavevmode\newline
The theory of canonical heights for \emph{morphisms} of~$\PP^N$ is
well known and may be developed exactly as was done by N\'eron and
Tate in their theory of canonical heights on abelian varieties; see
for example~\cite{callsilv:htonvariety} or~\cite[\S3.4]{MR2316407}.  A
regular affine automorphism~\cite{sibony:panoramas} is an automorphism
$\f:\AA^N\to\AA^N$ whose extension to a rational map
$\f:\PP^N\dashrightarrow\PP^N$ satisfies $Z(\f)\cap
Z(\f^{-1})=\emptyset$.  Regular affine automorphisms are algebraically
stable, i.e., $\d_\f=\deg(\f)$;
see~\cite[Chapter~2]{sibony:panoramas}. Shu Kawaguchi has developed a
theory of canonical heights for such maps.  Kawaguchi's construction
is described in~\cite{arxiv0405007}
and~\cite[Exercises~7.17--7.22]{MR2316407}, and the subtle height
inequality needed to justify the construction is given
in~\cite{arxiv0909.3573} and~\cite{arxiv0909.3107}.  There is thus a
satisfactory theory of canonical heights for regular affine
automorphisms, and the present article may be viewed as a first step
towards establishing an analogous theory for general dominant rational
maps.

\subsection{The dynamical Manin--Mumford conjecture}
\label{subsection:dynmanmum}
The relationship between preperiodic points and canonical heights, and
in particular Conjecture~\ref{conjecture:hzeroeqPPrePer}, may have
some bearing on the not-yet-pre\-cisely-for\-mu\-lated dynamical
Manin--Mumford conjecture. A naive conjecture, modeled after Zhang's
conjecture for polarized morphisms, might say the following: Let
$\f:\PP^N\dashrightarrow\PP^N$ be a dominant rational map with
$\d_\f>1$, and let~$X\subset\PP^N$ be an irreducible subvariety.  If
$\PrePer(\f)\cap X$ is Zariski dense in~$X$, then~$X$ is preperiodic.
This naive statement is clearly false.  For example,
let~$\f:\PP^2\dashrightarrow\PP^2$ be
$\f\bigl([x,y,z]\bigr)=[x^2z,y^3,z^3]$ and take $X=\{x=y\}$. But some
carefully formulated dynamical Manin--Mumford statements have been
proven; see for example~\cite{arxiv0805.1560,arxiv0705.1954}.  Our
hope is that the existence of a canonical height characterizing
preperiodic points as being exactly those points having height zero
might be a helpful tool for proving Manin--Mumford type results for
more general maps.

\subsection{Integrability and arithmetic entropy}
\label{subsection:integrability}

The relationship between the degree growth of iterates of a rational
map and the existence of invariant fibrations or more general
geometric invariant structures is an area of intense activity in both
the mathematical and the physics literatures.  When an invariant
structure of a specified type exists, one often says that the map is
integrable, although there is not yet a precise general definition of
integrability.  We refer the reader to~\cite{MR2556644} for a survey
on integrability of discrete dynamical systems and for some
(heuristic) methods of detecting integrability, including studying the
cycle structure of the reduction of~$\f$ acting on~$\PP^N(\FF_q)$ for
varying finite fields~$\FF_q$~\cite{MR2017601,MR2164737,MR2525820} and
studying the growth rate of~$h\bigl(\f^n(P)\bigr)$ for rational or
algebraic points~$P$~\cite{MR1732080,MR2131425,MR2212057}.  In
particular, Halburd~\cite{MR2131425} defines a map~$\f$ to be
\emph{Diophantine integrable} if~$h\bigl(\f^n(P)\bigr)$ grows no
faster than polynomially in~$n$ for all rational (or all algebraic)
points~$P$. In our terminology, such orbits have arithmetic
degree~$1$, equivalently, arithmetic entropy~$0$.  We also mention
Buium's beautiful arithmetic characterization~\cite{MR2171197} of
Latt\`es maps (one-dimensional integrable maps) in terms of their
mod~$p$ reductions.

\subsection{Another type of algebraic entropy}
\label{subsection:otheralgent}
There is another notion of algebraic entropy defined for self-maps of
topological groups with various additional structures, e.g., for
locally compact abelian groups. See for example the
papers~\cite{MR0175106,arxiv1111.1287,MR2491886,MR637984,MR540634,virili11}.
In particular, the paper~\cite{arxiv1111.1287} shows that the
algebraic entropy of an endomorphism of a finite-dimensional rational
vector space is the Mahler measure of the characteristic polynomial of
the associated matrix, which is similar to results of
Hasselblatt--Propp~\cite{MR2358970} and
Lin~\cite{arxiv1010.6285,arxiv1007.0253}.

\section{The dynamical degree of a rational map}
\label{section:dyndeg}
Let $\f:\PP^N\dashrightarrow\PP^N$ be a dominant rational map of
degree~$d$. If~$\f$ is a morphism, then the degree of~$\f^n$ is
simply~$d^n$, but in general the degree of~$\f^n$ may be strictly
smaller than~$d^n$. The sequence of degrees~$(\deg\f^n)_{n\ge1}$
is both interesting and often surprisingly difficult to analyze.

\begin{definition}
Let $\f:\PP^N\dashrightarrow\PP^N$ be a dominant rational map.
The (first) \emph{dynamical degree of~$\f$} is the quantity
\[
  \d_\f = \lim_{n\to\infty} (\deg\f^n)^{1/n}.
\]
\end{definition}

\begin{example}
\label{example:Fibondegmap2}
The iterates of the map
\[
\f\bigl([x,y,z]\bigr)=[yz,xy,z^2] 
\]
are easily computed to be
\[
  \f^n\bigl([x,y,z]\bigr)
  = \left[x^{F_{n-1}}y^{F_n}z^{F_n},x^{F_n}y^{F_{n+1}},z^{F_{n+2}}\right],
\]
where $F_n$ is the $n$'th Fibonacci number. Hence
\[
  \d_\f = \lim_{n\to\infty} F_{n+2}^{1/n} = \frac{1+\sqrt5}{2}
\]
is the golden ratio.  We remark that~$\f$ is birational and regular,
i.e., satisfies $Z(\f)\cap Z(\f^{-1})=\emptyset$, but it is not
an affine automorphism.
\end{example}

\begin{example}
\label{example:Fibondegmap}
The map
\[
  \f:\AA^3\longrightarrow\AA^3,\qquad
  \f(x,y,z)=(y,z,x+yz),
\]
is an affine automorphism, but it is not regular, since $Z(\f)\cap
Z(\f^{-1})$ is a line.  An easy induction shows that $\deg(\f^n) = F_n$,
so this map also has $\d_\f = \frac{1+\sqrt{5}}{2}$.
\end{example}

For the convenience of the reader, we recall the proof of the
following well-known properties of the dynamical degree.

\begin{proposition}
\label{proposition:df=inf}
The limit defining the dynamical degree exists and satisfies
\[
  \d_\f = \inf_{n\ge1} (\deg\f^n)^{1/n}.
\]
\end{proposition}
\begin{proof}
We note that for any rational maps~$\f,\psi:\PP^N\dashrightarrow\PP^N$, we
have
\begin{equation}
  \label{eqn:degfpsi}
  \deg(\f\circ\psi) \le (\deg\f)(\deg\psi).
\end{equation}
To ease notation, we let
\[
  d_n = \log\deg(\f^n).
\]
We need to prove that the sequence~$d_n/n$ converges and is equal to
its infimum. From~\eqref{eqn:degfpsi} we see that
\[
  d_{i+j} \le d_i + d_j \quad\text{for all $i$ and $j$.}
\]
Fix an integer~$m$ and write $n=mq+r$ with $0\le r<m$. Then
\[
  \frac{d_n}{n}
  =\frac{d_{mq+r}}{n}
  \le \frac{qd_m+d_r}{n}
  = \frac{d_m}{m}\cdot\frac{1}{1+r/mq}+\frac{d_r}{n}
  \le \frac{d_m}{m}+\frac{d_r}{n}.
\]
Now take the limsup as $n\to\infty$, keeping in mind that $m$ is
fixed and $r<m$, so~$d_r$ is bounded. This gives
\[
  \limsup_{n\to\infty} \frac{d_n}{n} \le \frac{d_m}{m}.
\]
Taking the infimum over~$m$ shows that
\[
  \limsup_{n\to\infty} \frac{d_n}{n}
  \le \inf_{m\ge1} \frac{d_m}{m}
  \le \liminf_{m\to\infty} \frac{d_m}{m},
\]
and hence all three quantities must be equal.
\end{proof}

The dynamical degrees in Examples~\ref{example:Fibondegmap2}
and~\ref{example:Fibondegmap} are the golden ratio, which is an
algebraic integer. This is a consequence of the fact that their degree
sequences~$(\deg\f^n)_{n\ge1}$ satisfy a linear recurrence with
constant coefficients.  It turns out that not all degree sequences
satisfy such linear recurrences. For example, it is shown
in~\cite{MR2358970} that the degree sequence for the map
$\f(x,y)=(xy^2,x^{-2}y)$ does not satisfy a linear recurrence with
constant coefficients, although it is still true that~$\d_\f$ is an
algebraic integer for this map.

\begin{conjecture}
\textup{(Bellon--Viallet \cite{MR1704282})}
\index{Bellon, M.P.}%
\index{Viallet, C.-M.}%
Let~$\f:\PP^N\dashrightarrow\PP^N$ be a dominant rational map defined
over~$\CC$. Then its dynamical degree~$\d_\f$ is an algebraic integer.
\end{conjecture}

\begin{example}
\label{example:ellmanyvalues}
If $\d_\f>1$, then $\deg(\f^n)$ grows roughly like $\d_\f^n$. One might
ask if the growth rates are the same, but it can happen that the ratio
$(\deg\f^n)/\d_\f^n$ grows like a power of~$n$.  For example, 
let~$d\ge2$ be an integer, and
let~$\f:\PP^N\dashrightarrow\PP^N$ be the dominant rational map given
in affine coordinates by
\[
  \f = \left( X_1^d X_2, X_2^d X_3, X_3^d X_4, \cdots, 
              X_{N-1}^d X_N, X_N^d  \right).
\]
(This is an example of a monomial map;
see Section~\ref{section:monomialmaps}.) It is easy to prove that
\[
  \deg(\f^n) = d^n + nd^{n-1} + \binom{n}{2}d^{n-2}
          + \cdots + \binom{n}{N-1}d^{n-N+1}.
\]
Thus
\[
  \lim_{n\to\infty} \frac{\deg(\f^n)}{d^n\cdot n^{N-1}} = \frac{1}{(N-1)!d^{N-1}},
\]
so Conjecture~\ref{conjecture:dyndegtimespoly} is true for this map
with $\d_\f=d$ and $\ell_\f=N-1$.  Trivial modifications of this
example give maps with $\d_\f=d$ and with~$\ell_\f$ equal to any
integer between~$0$ and~$N-1$.
\end{example}

\begin{remark}
We remark that more generally, a dominant rational map
$\f:\PP^N\dashrightarrow\PP^N$ has~$N$ different associated dynamical
degrees corresponding to its action on linear subspaces of various
dimensions.  Thus the \emph{$k^{\textup{th}}$ dynamical degree} of~$\f$ is
the quantity
\[
  \d_{k,\f} = \limsup_{n\to\infty} \bigl(\deg(\f^n)^*L\bigr)^{1/n},
\]
where~$L\subset\PP^N$ is a generic linear subvariety of
codimension~$k$. These dynamical degrees were introduced
in~\cite{MR1488341}, and are computed for monomial maps
in~\cite{arxiv1011.2854,arxiv1010.6285}.
\end{remark}

For further material on the dynamical degree, see for example
\cite{MR1139553,
MR2111418,
MR2428100,
MR1704282,
MR1877828,
MR1721516,
MR1867314,
MR2180409,
MR2339287,
arxiv1011.2854,
MR2358970,
arxiv1007.0253,
arxiv1010.6285, 
MR1836434,
MR2273670,
MR1488341,
MR2497321}

\section{Arithmetic degree}
\label{section:arithmeticdegree}

If $\f:\PP^N\to\PP^N$ is a morphism of degree~$d\ge2$ defined
over~$\Qbar$ and $P\notin\PrePer(\f)$, then $h\bigl(\f^n(P)\bigr)$
grows like a multiple of~$d^n$, so in particular
$h\bigl(\f^n(P)\bigr)^{1/n}\to d$ as $n\to\infty$. If~$\f$ is a
rational map, but not a morphism, then $h\bigl(\f^n(P)\bigr)$ may grow
more slowly than~$d^n$, which suggests (by analogy with dynamical
degree) the following definition.

\begin{definition}
Let $\f:\PP^N\dashrightarrow\PP^N$ be a dominant rational map defined
over~$\Qbar$, and let~$P\in\PP^N(\Qbar)_\f$.  The \emph{arithmetic
  degree of~$\f$ at~$P$} is the quantity
\[
  \a_\f(P) = 
      \limsup_{n\to\infty} h\bigl(\f^n(P)\bigr)^{1/n}.
\]
(If $h\bigl(\f^n(P)\bigr)=0$ for all sufficiently large~$n$, which can
only happen if~$P\in\PrePer(\f)$, then by convention we 
set~$\a_\f(P)=1$.) 
\end{definition}

\begin{example}
Consider the map given in affine coordinates by
\[
  \f(x,y,z)=(xy,y,z^2).
\]
Then
\[
  \f^n(x,y,z) = \bigl( xy^{n},y,z^{2^n} \bigr),
\]
so $\d_\f=2$, since the $z$-coordinate 
dominates the degree of~$\f^n$. However, if we consider a
point of the form $P=(x,y,\z)$ with $\z$ a root of unity, then
\[
  h\bigl(\f^n(x,y,\z)\bigr) \le h(x)+nh(y),
\]
so $\a_\f(P)=1$. Thus $\a_\f(P)=1$ on a countable union of
two-dimen\-sional hyperplanes of~$\PP^3$. These hyperplanes are
preperiodic for~$\f$.
\end{example}

We now show that $\a_\f(P)\le\d_\f$.

\begin{proposition}
\label{proposition:hfnPledfnhP}
Let $\f:\PP^N\dashrightarrow\PP^N$ be a dominant rational map of
degree~$d\ge2$ defined over~$\Qbar$, and let~$P\in\PP^N(\Qbar)_\f$.
Then
\[
  \a_\f(P) \le \d_\f.
\]
\end{proposition}
\begin{proof}
A standard triangle inequality estimate says that 
\begin{equation}
  \label{eqn:hfQledhQ}
  h\bigl(\f(Q)\bigr) \le d h(Q) + O_\f(1).
\end{equation}
\par
We start with a telescoping sum inequality computation.
\begin{align*}
  h\bigl(\f^n(P)\bigr) - d^nh(P)
  &= \sum_{i=1}^n d^{n-i}\left[h\bigl(\f^i(P)\bigr)
          - d h\bigl(\f^{i-1}(P)\bigr) \right] \\
  &\le \sum_{i=1}^n d^{n-i}\sup_{Q\in\PP^N(\Qbar)}\left\{h\bigl(\f(Q)\bigr)
           - d h(Q) \right\} \\
  &\le \sum_{i=1}^n d^{n-i} O_\f(1)
   \quad\text{from \eqref{eqn:hfQledhQ},} \\
  &= O_\f(d^n).
\end{align*}
Thus there is a constant $C(\f,P)$, depending as indicated on~$\f$
and~$P$, such that
\begin{equation}
  \label{eqn:hfnPdnhPOfdn}
  h\bigl(\f^n(P)\bigr) \le C(\f,P)(\deg\f)^n
  \quad\text{for all $n\ge1$.}
\end{equation}
\par
For each integer $k\ge1$, we write $d_k=\deg(\f^k)$.  
Applying~\eqref{eqn:hfnPdnhPOfdn} to the map~$\f^k$ yields
\[
  h\bigl(\f^{nk}(P)\bigr) \le C(\f^k,P)(\deg\f^k)^n.
\]
Hence
\begin{equation}
  \label{eqn:limsuphfnkP}
  \limsup_{n\to\infty} h\bigl(\f^{nk}(P)\bigr)^{1/nk}
  \le \limsup_{n\to\infty} C(\f^k,P)^{1/nk}(\deg\f^k)^{1/k}
  = (\deg\f^k)^{1/k}.
\end{equation}
\par
We next show that~$\a_\f(P)$ can be computed using the subsequence of
iterates~$(\f^{nk})_{n\ge1}$. To see this, we estimate
\begin{align*}
  \a_\f(P)
  &= \limsup_{m\to\infty} h\bigl(\f^m(P)\bigr)^{1/m} \\
  &= \limsup_{n\to\infty} \max_{0\le i<k} h\bigl(\f^{nk+i}(P)\bigr)^{1/(nk+i)} \\
  &\le \limsup_{n\to\infty} \max_{0\le i<k} 
     \Bigl(d^ih\bigl(\f^{nk}(P)\bigr)+O(d^i)\Bigr)^{1/(nk+i)} 
   \quad\text{from \eqref{eqn:hfQledhQ},}\\
  &\le \limsup_{n\to\infty} 
     \Bigl(d^{k-1}h\bigl(\f^{nk}(P)\bigr)+O(d^k)\Bigr)^{1/nk}  \\
  &=  \limsup_{n\to\infty} h\bigl(\f^{nk}(P)\bigr)^{1/nk}  \\
  &\le \a_\f(P) \quad\text{by definition of $\a_\f(P)$.}
\end{align*}
This proves that for any integer~$k\ge2$, the arithmetic degree of~$P$
can be computed as
\begin{equation}
  \label{eqn:afPlimsuphfnk}
  \a_\f(P) = \limsup_{n\to\infty} h\bigl(\f^{nk}(P)\bigr)^{1/nk}.
\end{equation}
\par
Combining~\eqref{eqn:limsuphfnkP} and~\eqref{eqn:afPlimsuphfnk} gives
\[
  \a_\f(P) \le (\deg\f^k)^{1/k}.
\]
This estimate holds for all $k\ge1$, so letting $k\to\infty$ gives
the desired result $\a_\f(P)\le \d_\f$.
\end{proof}

\begin{question}
If~$P\in\PrePer(\f)$, then~$\a_\f(P)=1$, while
Proposition~\ref{proposition:hfnPledfnhP} says that
$\a_\f(P)\le\d_\f$. The arithmetic degree can thus be used to stratify
the points in~$\PP^N(\Qbar)$. Conjecture~\ref{conjecture:setofafPs}(a)
says that~$\a_\f(P)$ takes on only finitely many values.  What
do sets of the form
\begin{equation}
  \label{eqn:setafPeqb}
  \bigl\{Q\in\PP^N(\Qbar)_\f: \a_\f(Q)=\a_\f(P)\bigr\}
\end{equation}
look like for the finitely many possible values of~$\a_\f(P)$? 
\end{question}

\section{Canonical heights for dominant rational maps}
\label{section:canonicalht}

In this section we define and study basic properties of canonical
heights for general dominant rational maps. Later we give refined
results for monomial maps.

\begin{definition}
Let $\f:\PP^N\dashrightarrow\PP^N$ be a dominant rational map defined
over~$\Qbar$ with dynamical degree $\d_\f$ and associated quantity
$\ell_\f$ as defined in Conjecture~\ref{conjecture:dyndegtimespoly}.
Assume that $\d_\f>1$.  Let $P\in\PP^N(\Qbar)_\f$.  The
\emph{canonical height of~$P$} (\emph{relative to~$\f$}) is the
quantity
\[
  \hhat_\f(P)=\limsup_{n\to\infty} \frac{1}{n^{\ell_\f}\d_\f^{n}}h\bigl(\f^n(P)\bigr).
\]
\end{definition}

We give an example to show that the limsup is necessary in the
definition of the canonical height.

\begin{example}
\label{example:hhatnotlimit}
Let $d\ge2$ be an integer, and let $\f:\PP^2\dashrightarrow\PP^2$ be
the map $\f(x,y)=(x^{-d},y^{-d})$.  In homogeneous coordinates, we
have
\[
  \f^n\bigl([X,Y,1]\bigr) = \begin{cases} 
       \left[X^{d^n},Y^{d^n},1\right]&\text{if $n$ is even,}\\
       \left[Y^{d^n},X^{d^n},X^{d^n}Y^{d^n}\right]&\text{if $n$ is odd.}\\
  \end{cases}
\]
Thus $\deg(\f^n)=d^n$ if~$n$ is even, and~$\deg(\f^n)=2d^n$ if~$n$ is odd,
so in particular $\d_\f=d$ and $\ell_\f=0$.
\par
We now consider points $P=(x,y)\in\ZZ^2$ with $xy\ne0$. Then
\[
  d^{-n}h\bigl(\f^n(P)\bigr) = \begin{cases}
     \log\max\bigl\{|x|,|y|\bigr\}&\text{if $n$ is even,}\\
     \log\bigl(|xy|\bigr)&\text{if $n$ is odd.}\\
  \end{cases}
\]
Thus the sequence $\d_\f^{-n}h\bigl(\f^n(P)\bigr)$ does not have
a limit (unless $|x|=1$ or $|y|=1$). 
\end{example}

\begin{question}
Is it true that the sequence
$n^{-\ell_\f}\d_\f^{-n}h\bigl(\f^n(P)\bigr)$ has only finitely many
accumulation points in~$\RR\cup\{\infty\}$?
\end{question}

Our next example shows that the $n^{\ell_\f}$ factor in the definition
of~$\hhat_\f$ is necessary if we want the canonical height to be
finite.

\begin{example}
\label{example:hteqinfty}
Consider the map $\f([1,x,y])=[1,x^dy,y^d]$ with $d\ge2$. Then
\[
  \f^n([1,x,y]) = \left[1,x^{d^n}y^{nd^{n-1}},y^{d^n}\right],
\]
so
\[
  \d_\f = \lim_{n\to\infty} (d^n+nd^{n-1})^{1/n} = d
  \quad\text{and}\quad
  \ell_\f = 1.
\]
Then for integers~$x$ and~$y$  with $xy\ne0$, we have
\begin{align*}
  \hhat_\f([1,x,y])
  &= \limsup_{n\to\infty} 
      n^{-1} d^{-n} h\left(\left[1,x^{d^n}y^{nd^{n-1}},y^{d^n}\right]\right)\\
  &= \limsup_{n\to\infty} n^{-1}d^{-n} \log\left|x^{d^n}y^{nd^{n-1}}\right|\\
  &= \limsup_{n\to\infty} \left(\frac{1}{n}\log|x|
      + \frac{1}{d}\log|y|\right)
  = \frac{1}{d}\log|y|.
\end{align*}
\end{example}

Unfortunately, as the next example shows, the assumption that
$\ell_\f>0$ does not suffice to imply that $\hhat_\f(P)$ is finite.

\begin{example}
\label{example:infiniteht}
Let~$\f:\PP^3\to\PP^3$ be the map given in affine coordinates by
\[
  \f(x,y,z) = (xy+xz,y+z,z).
\]
Then
\[
  \f^n(x,y,z)=\bigl(x(y+z)(y+2z)\cdots(y+nz),y+nz,z\bigr),
\]
so
\[
  \deg(\f^n)=n+1,\quad \d_\f=1,\quad\text{and}\quad \ell_\f=1.
\]
On the other hand, we have
\[
  \frac{h\bigl(\f^n(1,0,1)\bigr)}{n^{\ell_\f}\d_\f^n}
  = \frac{h(n!,n,1)}{n}
  = \frac{\log (n!)}{n}
  \sim \log n
  \quad\text{as $n\to\infty$.}
\]
Hence for this example we have~$\hhat_\f(1,0,1)=\infty$.
\end{example}

\begin{question}
\label{question:hhatfinite}
If~$\d_\f>1$, is it true that the canonical height~$\hhat_\f(P)$ is
finite? Example~\ref{example:infiniteht} shows that the answer is
negative if~$\d_\f=1$, even if we require that~$\ell_\f>0$.
\end{question}

In Section~\ref{section:monomialmaps} we prove that
Question~\ref{question:hhatfinite} has an affirmative answer for
monomial maps; see Proposition~\ref{proposition:adjhtformonmaps}.

\begin{proposition}
\label{proposition:canhtprops}
The canonical height has the following properties\textup:
\begin{parts}
\Part{(a)}
$0\le \hhat_\f(P)\le\infty$.
\Part{(b)}
$\hhat_\f\bigl(\f(P)\bigr)=\d_\f\hhat_\f(P)$.
\Part{(c)}
If $P\in\PrePer(\f)$, then $\hhat_\f(P)=0$.
\Part{(d)}
If $\hhat_\f(P)>0$, then $\a_\f(P)=\d_\f$.
\end{parts}
\end{proposition}
\begin{proof} (a) 
This is obvious, since the height~$h$ is
a non-negative function. 
\par\noindent(b)\enspace
We compute
\begin{align*}
  \hhat_\f\bigl(\f^n(P)\bigr)
  &=\limsup_{n\to\infty} \frac{1}{n^{\ell_\f}\d_\f^{n}}h\bigl(\f^{n+1}(P)\bigr) \\
  &=\limsup_{n\to\infty} \frac{1}{(n-1)^{\ell_\f}\d_\f^{n-1}}h\bigl(\f^{n}(P)\bigr) \\
  &=\d_\f\limsup_{n\to\infty} 
      \left(\frac{n}{n-1}\right)^{\ell_\f}  
      \frac{1}{n^{\ell_\f}\d_\f^{n}}h\bigl(\f^{n}(P)\bigr) \\
  &=\d_\f\hhat_\f(P).
\end{align*}
\par\noindent(c)\enspace
If $P$ is preperiodic, then $h\bigl(\f^n(P)\bigr)$ takes on only finitely
many values, so it is immediate from the defintion of~$\hhat_\f$
that $\hhat_\f(P)=0$.
\par\noindent(d)\enspace 
We are assuming that $\hhat_{\f}(P)>0$, and by definition
$\hhat_\f(P)$ is the limsup of
$n^{-\ell_\f}\d_\f^{-n}h\bigl(\f^n(P)\bigr)$, so we can find an infinite
sequence~$\Ncal$ of positive integers such that 
\[
  n^{-\ell_\f}\d_\f^{-n} h\bigl(\f^n(P)\bigr) \ge \frac12\hhat_\f(P) > 0
  \quad\text{for all $n\in\Ncal$.}
\]
It follows that
\[
  \a_{\f}(P)
  = \limsup_{n\to\infty} h\bigl(\f^n(P)\bigr)^{1/n} 
 \ge \limsup_{n\in\Ncal} 
   \left(n^{\ell_\f}\d_\f^{n}\cdot\frac12\hhat_{\f}(P)\right)^{1/n} 
 = \d_{\f}.
\]
But we know from Proposition~\ref{proposition:hfnPledfnhP} that
$\a_\f(P)\le\d_\f$ for every dominant rational map~$\f$, so this proves
that $\a_{\f}(P)=\d_{\f}$.  
\end{proof}

\begin{remark}
The implication
\[
  P\in\PrePer(\f)\quad\Longrightarrow\quad  \hhat_{\f}(P)=0
\]
in Proposition~\ref{proposition:canhtprops}(c) is trivial, but for
applications one generally wants to know that the opposite implication
holds, at least off of an explicitly described exceptional set. One
way to prove the opposite implication is to show that
$\hhat_{\f}(P)$ is equal to $h(P)+O(1)$, or at least satisfies
$\hhat_\f(P)\asymp h(P)$, again off of an exceptional
set.\footnote{For nonnegative functions~$F$ and~$G$, we write
  $F\asymp{G}$ to mean that there are positive
  constants~$c_1,c_2,c_3,c_4$ such that $c_1F(x)-c_2\le G(x)\le
  c_3F(x)+c_4$.}  From such an estimate, it immediately follows that
\begin{align*}
  \hhat_{\f}(P)=0
  &\quad\Longrightarrow\quad 0=\d_\f^n\hhat_\f(P)=\hhat_\f\bigl(\f^n(P)\bigr) \\
  &\quad\Longrightarrow\quad h\bigl(\f^n(P)\bigr) \asymp 0 \\
  &\quad\Longrightarrow\quad
        \text{$\Orbit_\f(P)$ is a set of bounded height,} \\
  &\quad\Longrightarrow\quad \text{$\Orbit_\f(P)$ is a finite set.}
\end{align*}
\par
When $\f$ is a birational map, another method used to prove the reverse
implication is to use an estimate of the form
\[
  C_1h\bigl(\f(P)\bigr) + C_2h\bigl(\f^{-1}(P)\bigr) \ge C_3 h(P) + O(1);
\]
see for example~\cite{denis:periodicaffineaut, arxiv0405007, arxiv0909.3573,
arxiv0909.3107, MR1988948, silverman:K3heights, silverman:henonmap}
for results of this type for regular affine automorphisms.
\par
Corollary~\ref{corollary:irredcharpoly} says that the reverse
implication
\[
  \hhat_{\f}(P)=0 \quad\Longrightarrow\quad  P\in\PrePer(\f)
\]
is true for a certain (large) class of monomial maps on~$\PP^N$, but
the proof is not via an estimate $\hhat_\f(P)\asymp h(P)$.
\end{remark}

\section{Monomial maps and canonical heights}
\label{section:monomialmaps}

A monomial map is an endomorphism of the
torus~$\GG_m^N$. Embedding~$\GG_m^N$ in~$\PP^N$, monomial maps induce
rational self-maps of~$\PP^N$. In this section we study the geometry
of iteration of these maps and prove that the canonical height is
finite. We begin with a formal definition which sets the notation that
we will use throughout the rest of this article.

\begin{definition}
We write $\Mat_N^+(\ZZ)$ for the set of $N$-by-$N$ matrices with
integer coefficients and nonzero determinant. To each matrix
$A\in\Mat_N^+(\ZZ)$ we associate the
\emph{monomial map} $\f_A:\GG_m^N\to\GG_m^N$ given by the formula
\begin{multline*}
  \f_A(X_1,\ldots,X_N)={}\\
   \left(
    X_1^{a_{11}}X_2^{a_{12}}\cdots X_N^{a_{1N}},\;
    X_1^{a_{21}}X_2^{a_{22}}\cdots X_N^{a_{2N}},\ldots,
    X_1^{a_{N1}}X_2^{a_{N2}}\cdots X_N^{a_{NN}} \right).
\end{multline*}
We call~$\f_A$ the \emph{monomial map associated to~$A$}.  We note
that~$\f_A$ induces a rational map $\f_A:\PP^N\dashrightarrow\PP^N$.
We denote the \emph{spectral radius} of~$A$ by
\[
  \rho(A) = \max\bigl\{|\l| : 
    \text{$\l\in\CC$ is an eigenvalue for $A$}\bigr\}.
\]
\end{definition}

It is immediate from the definition that if $A,B\in\Mat_N^+(\ZZ)$ are
matrices with associated monomial maps~$\f_A$ and~$\f_B$, then
\begin{equation}
  \label{eqn:fABeqfAfB}
  \f_{AB}(P) = (\f_A\circ\f_B)(P)
  \quad\text{and}\quad
  \f_{A+B}(P) = \f_A(P)\cdot\f_B(P).
\end{equation}

\begin{proposition}
\label{proposition:monmapprops}
Let~$A\in\Mat_N^+(\ZZ)$ be a matrix with associated monomial map~$\f_A$.%
\begin{parts}
\Part{(a)}
$\rho(A)\ge1$.
\Part{(b)}
$\rho(A)=1$ if and only if all of the eigenvalues of~$A$ are roots of unity,
which is equivalent to
$(A^m-I)^n=0$ for some positive integers~$m$ and~$n$.
\Part{(c)}
\textup{(Hasselblatt--Propp~\cite{MR2358970}; see also~\cite{arxiv1010.6285})}
The dynamical degree of~$\f_A$ is equal to its spectral radius,
\[
  \d_{\f_A}=\rho(A).
\]
\Part{(d)}
If none of the eigenvalues of~$A$ are roots of unity, then
the set of preperiodic points of~$\f_A$ in~$\GG_m^N(\CC)$ is
\[
  \PrePer(\f_A) = \GG_m^N(\Qbar)_\tors,
\]
where the torsion subgroup~$\GG_m^N(\Qbar)_\tors$ is the set of 
points whose coordinates are roots of unity.
\end{parts}
\end{proposition}
\begin{proof}
Let $\l_1,\ldots,\l_N$ be the eigenvalues of~$A$, labeled so that
$|\l_1|=\rho(A)$. The product~$\l_1\l_2\cdots\l_N$ of the eigenvalues
equals $\det(A)$, which is a non-zero integer, so
certainly~$|\l_1|\ge1$. This proves~(a)
\par
If~$\rho(A)=1$, then for every~$i$, the algebraic integer~$\l_i$ has
the property that all of its Galois conjugates are in the closed unit
circle. It follows from Kronecker's
theorem~\cite[Theorem~3.8]{MR2316407} that~$\l_i$ is a root of
unity. Thus all of the eigenvalues of~$A$ are roots of unity, so~$A$
is quasi-unipotent. Conversely, if~$A$ is quasi-unipotent, then its
characteristic polynomial divides $(T^n-1)^m$ for some $n\ge1$ and
$m\ge1$, so the eigenvalues of~$A$ are roots of unity, hence have
absolute value equal to~$1$. This proves the
first part of~(b), and the second part is easy.
\par
The fact that $\d_{\f_A}=\rho(A)$ is due to Hasselblatt and
Propp~\cite[Theorem~6.2]{MR2358970}, which gives~(c).
\par
Finally, for~(d), we use~\eqref{eqn:fABeqfAfB} to see that
\[
  \f_A^n(P)=\f_A^m(P)
  \;\Longleftrightarrow\;
  \f_{A^n}(P)=\f_{A^m}(P)
  \;\Longleftrightarrow\;
  \f_{A^n-A^m}(P)=1.
\]
Thus~(b) and the following lemma complete the proof of~(d).
\end{proof}

\begin{lemma}
\label{lemma:fAPeqfBP} 
Let~$B\in\Mat_N^+(\ZZ)$ be a matrix with~$\det(B)\ne0$, and
suppose that a point $P\in\GG_m^N(\CC)$ satisfies
\[
  \f_B(P) = 1.
\]
Then every coordinate of~$P$ is a root of unity.  
\end{lemma}
\begin{proof}
For notational clarity, we write $\bfe=(1,1,\ldots,1)$ for the
identity element of $\GG_m^N(\CC)$, so our assumption
is that~$\f_B(P)=\bfe$.  
Let $\D=\det(B)$, and let~$C=B^{\textup{adj}}$ be the adjoint
matrix, so $CB=\D I_N$. Then
\[
  \bfe = \f_C(\bfe)
  = \f_C\circ\f_{B}(P)
  = \f_{CB}(P)
  = \f_{\D I_N}(P).
\]
Hence every coordinate of~$P$ is a $\D^{\text{th}}$-root of unity.  [I
thank Jan-Li Lin (private communication) for showing me this proof.
More generally, for any~$B\in\Mat_N(\ZZ)$, the map $\f_B$ is an
endomorphism of~$\GG_m^N$, so its kernel is an algebraic subgroup
with codimension equal to the rank of~$B$.]
\end{proof}

\begin{remark}
The restriction in Proposition~\ref{proposition:monmapprops}(d)
that~$A$ has no eigenvalues that are roots of unity is necessary, as
is seen for example from the map
\[
  \f(x,y)=(x^ay^{1-a},x^by^{1-b})
  \quad\text{associated to the matrix}\quad
  A=\left(\begin{smallmatrix} a&1-a\\b&1-b\\ \end{smallmatrix}\right).
\]
Then~$1$ is an eigenvalue of~$A$, and $\f(t,t)=(t,t)$, so~$(t,t)$ is a
fixed point for every~$t$.
\end{remark}

Lin, Jonsson, and Wulcan have proven a strengthened version of
Conjecture~\ref{conjecture:dyndegtimespoly} for monomial maps.

\begin{theorem}
\label{theorem:ellinvariantmonomap}
\textup{(Lin \cite[Theorem~6.2]{arxiv1007.0253},
Jonsson--Wulcan~\cite{arXiv:1001.3938})} Let
$A\in\Mat_N^+(\ZZ)$, let~$\f_A$ be the associated monomial map, and
let~$\ell(A)+1$ be the dimension of the largest Jordan block of~$A$
among those blocks corresponding to eigenvalues of maximal absolute
value.  Then 
\begin{equation}
  \label{eqn:degfAnasympnellA}
  \deg(\f_A^n) \asymp n^{\ell(A)} \rho(A)^n
  \quad\text{for all $n\ge1$,}
\end{equation}
where the implied constants depend only on~$A$.  In
particular,~$\ell(A)$ is an integer satisfying $0\le\ell(A)<N$.
\end{theorem}

Lin proves Theorem~\ref{theorem:ellinvariantmonomap} by extending the
degree map $A\mapsto\deg(\f_A)$ to a function on~$\Mat_N^+(\RR)$,
showing that the resulting function is more-or-less a norm, and using
a compactness argument.  We can use
Theorem~\ref{theorem:ellinvariantmonomap} and an elementary argument
to prove that the canonical height for monomial maps is finite.

\begin{proposition}
\label{proposition:adjhtformonmaps}
Let~$\f_A$ be a monomial map associated to a
matrix~$A\in\Mat_N^+(\ZZ)$ with either $\rho(A)>0$ or $\ell(A)>0$,
where~$\rho(A)$ is the spectral radius of~$A$ and~$\ell(A)$ is as in
the statement of Theorem~$\ref{theorem:ellinvariantmonomap}$.  Then
there is a constant~$C(A)$ such that
\[
  \hhat_{\f_A}(P) \le C(A)h(P) 
  \quad\text{for all~$P\in\GG_m^N(\Qbar)$.}
\]
In particular, the canonical height $\hhat_{\f_A}(P)$ is finite.
\end{proposition}
\begin{proof}
Theorem~\ref{theorem:ellinvariantmonomap} and the definition
of dynamical degree imply that
$\d_\f=\rho(A)$ and $\ell_\f=\ell(A)$, so the definition
of the canonical height becomes
\begin{equation}
  \label{eqn:hhatfAPnlrhoA}
  \hhat_{\f_A}(P) = \limsup_{n\to\infty}
  \frac{1}{n^{\ell(A)}\rho(A)^n}h\bigl(\f^n(P)\bigr).
\end{equation}
For a matrix $C=(c_{ij})\in\Mat_N(\CC)$, we write $\|C\|_\infty$ for
the sup-norm $\max|c_{ij}|$.  An elementary triangle inequality
estimate (cf.\ \cite[Lemma~6.4]{arxiv1007.0253}) gives
\begin{equation}
  \label{eqn:AnllnellArAn}
  \|A^n\|_\infty \ll n^{\ell(A)}\rho(A)^n.
\end{equation}
(In fact, $\|A^n\|_\infty \asymp n^{\ell(A)}\rho(A)^n$.)
\par
We  write
\[
  A^n = \bigl(a_{ij}(n)\bigr)_{1\le i,j\le N},
\]
and we let
\[
  P=(x_1,\ldots,x_N)\quad\text{and}\quad
  \f_A^n(P)=(y_1,\ldots,y_N).
\]
Then
\begin{align*}
  h\bigl(\f^n(P)\bigr)
  &= \sum_{v\in M_K} 
     {\max_{1\le i\le N}} \log^+\|y_i\|_v \\
  &= \sum_{v\in M_K} 
     {\max_{1\le i\le N}} \left\{0,
       \sum_{j=1}^N a_{ij}(n)\log\|x_i\|_v \right\} \\
  &\le N \|A^n\|_\infty \sum_{v\in M_K} {\max_{1\le i\le N}} \log^+\|x_i\|_v \\
  &\ll  n^{\ell(A)}\rho(A)^nh(P)
      \quad\text{from \eqref{eqn:AnllnellArAn},} 
\end{align*}
where the implied constant depends on~$N$ and~$A$, but is
independent of~$n$ and~$P$. This inequality shows that the limsup
in~\eqref{eqn:hhatfAPnlrhoA} is finite and bounded by a constant multiple
of~$h(P)$.
\end{proof}

\begin{remark}
We note that
\[
  \hhat_{\f_A}(x_1,\ldots,x_N)
  = \hhat_{\f_A}(\z_1x_1,\ldots,\z_Nx_N)
\]
for any roots of unity~$\z_1,\ldots,\z_N$.  In particular, the set of
points satisfying~$\hhat_{\f_A}(P)=0$ is invariant under multiplying
the coordinates of the points by roots of unity.
\end{remark}

\section{Points of canonical height zero for monomial maps}
\label{section:monomialmapht0}

It is a trivial fact (Proposition~\ref{proposition:canhtprops}(c))
that preperiodic points have canonical height zero. The converse is
not true in full generality, and it can be quite delicate to determine
the set of points having canonical height zero.  In this section we
prove, among other things, that the converse is true for monomial maps
whose associated matrix has irreducible characteristic polynomial. We
start with a general result which says that for monomial maps, the set
of points of canonical height zero lies in a proper algebraic
subgroup.

\begin{definition}
Let~$A\in\GL_N(\QQ)$.  A \emph{Jordan subspace} for~$A$ is
an~$A$-invariant subspace of~$\Qbar^N$ corresponding to a single
Jordan block of~$A$. A Jordan subspace~$V\subset\Qbar^N$ with associated
eigenvalue~$\l$ is called a \emph{maximal Jordan subspace} if
$|\l|=\rho(A)$ and if the dimension of~$V$ is maximal among the Jordan
subspaces whose eigenvalue have magnitude equal to~$\rho(A)$.  We set
\begin{align*}
  r(A) &= \text{number of maximal Jordan subspaces,} \\
  \rbar(A) &= \#\left\{ \s(V) : 
      \begin{tabular}{@{}l@{}}
           $V$ is a maximal Jordan subspace\\
           for $A$ and $\s\in\Gal(\Qbar/\QQ)$\\
      \end{tabular}
    \right\}.
\end{align*}
Thus~$\rbar(A)$ is the number of distinct $\Qbar$-subspaces
of~$\Qbar^N$ that are Galois conjugate to a maximal Jordan subspace
of~$A$. We note that~$\rbar(A)\ge r(A)\ge1$, since~$A$
always has  at least one maximal Jordan subspace.
\end{definition}

\begin{definition}
Let~$G$ be an algebraic subgroup of $\GG_m^N$.  We
write~$G(\Qbar)^\div$ for the divisible hull of~$G(\Qbar)$,
\begin{multline*}
  G(\Qbar)^\div
  = \bigl\{(\a_1,\ldots,\a_N)\in\GG_m^N(\Qbar) \\
    {} : \text{$(\a_1^n,\ldots,\a_N^n)\in G(\Qbar)$ for some $n\ge1$}\bigr\}.
\end{multline*}
Equivalently,~$G(\Qbar)^\div$ is the set of translates of~$G(\Qbar)$
by points in $\GG_m^N(\Qbar)_\tors$.\EndNote{appendix5}
\end{definition}

We can now state our main result.

\begin{theorem}
\label{theorem:linrelsonlogP}
Let $A\in\Mat_N^+(\ZZ)$ be a matrix whose associated monomial
map~$\f_A$ has dynamical degree~$\d_{\f_A}>1$. There is an algebraic
subgroup $G\subset\GG_m^N$ with dimension
\[
  \dim G \ge N - \rbar(A)
\]
such that
\[
  \bigl\{P\in\GG_m^N(\Qbar) : \hhat_{\f_A}(P)=0 \bigr\}
  \subset G(\Qbar)^\div.
\]
\end{theorem}

\begin{remark}
The proof of Theorem~\ref{theorem:linrelsonlogP} describes explicitly
how to construct the group~$G$ from the matrix~$A$.
\end{remark}

Theorem~\ref{theorem:linrelsonlogP} has a number of interesting
corollaries.

\begin{corollary}
\label{corollary:ht0orbitnotdense}
Let $\f:\GG_m^N\to\GG_m^N$ be a monomial map with $\d_\f>1$, and
let~$P\in\GG_m^N(\Qbar)$ be a point with $\hhat_\f(P)=0$.  Then there
is a proper algebraic subgroup $G\subsetneq\GG_m^N$ with
$\Orbit_\f(P)\subset G$.  In particular, the orbit~$\Orbit_{\f}(P)$ is
not Zariski dense in~$\GG_m^N$.
\end{corollary}

\begin{remark}
It should be possible to use an effective form of Baker's theorem to
prove effective versions of Theorem~\ref{theorem:linrelsonlogP} and
Corollary~\ref{corollary:ht0orbitnotdense}.  Thus for example, the
proofs should yield an effective constant $C=C\bigl(A,h(P)\bigr)>0$
such that
\[
  \text{$\Orbit_\f(P)$ Zarisiki dense}
  \quad\Longrightarrow\quad
  \hhat_{\f_A}(P)> C.
\]
Presumably the constant~$C$ computed in this way is very small if the
coefficients of~$A$ or the height of~$P$ is large.  It is an
interesting question as to whether a Lehmer-type estimate holds, e.g.,
is it possible to take $C=C'h(P)^{-k}$ for constants~$C'$
and~$k$ that depend only on~$A$?  (Maybe even with~$k$ depending only
on~$N$?)
\end{remark}

For monomial maps whose associated matrices have irreducible 
characteristic polynomial, we can say more. 

\begin{corollary}
\label{corollary:irredcharpoly}
Let~$\f:\GG_m^N\to\GG_m^N$ be a monomial map with dynamical
degree~$\d_\f>1$ defined by a matrix $A\in\Mat_N^+(\ZZ)$ whose
characteristic polynomial is irreducible over~$\QQ$. Let
$P\in\GG_m^N(\Qbar)$. Then
\[
  \hhat_{\f}(P)=0\quad\Longleftrightarrow\quad P\in\PrePer(\f).
\]
\end{corollary}

We next describe the set of arithmetic degrees for a monomial map.

\begin{definition}
For a polynomial~$f(T)\in\CC[T]$, we write
\[
  \rho(f) = \max\bigl\{|\a| : \text{$\a\in\CC$ is a root of $f$}\bigr\}.
\]
With this notation, the spectral radius of a matrix~$A\in\Mat_N(\CC)$
is $\rho(A)=\rho\bigl(\det(T-A)\bigr)$.
\end{definition}

\begin{corollary}
\label{corollary:algdegformonmap}
Let~$\f:\GG_m^N\to\GG_m^N$ be a monomial map defined by a matrix
$A\in\Mat_N^+(\ZZ)$, and let $f_1(T),\ldots,f_s(T)\in\ZZ[T]$ be the
monic irreducible factors of the characteristic polynomial~$\det(T-A)$
of~$A$.  Then
\[
  \bigl\{\a_{\f}(P) : P \in\GG_m^N(\Qbar)\bigr\}
  = \bigl\{1,\rho(f_1),\rho(f_2),\ldots,\rho(f_s)\bigr\}.
\]
In particular, for every $P\in\GG_m^N(\Qbar)$ the algebraic
degree~$\a_{\f}(P)$ is an algebraic integer, and~$\a_{\f}(P)$ takes on
only finitely many values as~$P$ ranges over~$\GG_m^N(\Qbar)$.
\end{corollary}

All rational maps have the property that if $\hhat_\f(P)>0$, then
$\a_\f(P)=\d_\f$; see Proposition~\ref{proposition:canhtprops}(d). For
monomial maps whose associated matrices are diagonalizable, we can prove
the converse. We note that some restriction is necessary, since it is
easy to construct non-diagonalizable monomial maps for which
Corollary~\ref{corollary:hgt0iffaeqd} is false. Indeed, the map in
Example~\ref{example:hteqinfty} applied to the point~$P=[1,2,1]$
provides an example with $\hhat_\f(P)=0$ and $\a_\f(P)=\d_\f$. (I thank
Mattias Jonsson for this last observation.)

\begin{corollary}
\label{corollary:hgt0iffaeqd}
Let $A\in\Mat_N^+(\ZZ)$ be a matrix that is diagonalizable over~$\CC$,
and assume that the associated monomial map
$\f:\PP^N\dashrightarrow\PP^N$ satisfies $\d_{\f}>1$. Let
$P\in\GG_m^N(\Qbar)$. Then
\[
  \hhat_{\f}(P)>0
  \quad\Longleftrightarrow\quad
  \a_{\f}(P)=\d_{\f}.
\]
\end{corollary}

\section{Proof of Theorem~\ref{theorem:linrelsonlogP}}
\label{section:proofmainthm}

In this section we give the  proof of Theorem~\ref{theorem:linrelsonlogP}.
We start with some additional notation.

\begin{definition}
By definition, all of the maximal Jordan subspaces for a matrix~$A$
have the same dimension. We let
\[
  \ell(A) = \dim(\text{any maximal Jordan subspace}) - 1.
\]
We remark that if~$\f_A$ is the monomial map associated to a
matrix~$A\in\Mat_N^+(\ZZ)$ and~$\ell_{\f_A}$ is the associated degree
growth exponent defined in
Conjecture~\ref{conjecture:dyndegtimespoly}, then
Proposition~\ref{theorem:ellinvariantmonomap} implies
that~$\ell(A)=\ell_{\f_A}$, so our use of~$\ell$ for two seemingly
different purposes is consistant.
\end{definition}

\begin{definition}
Let $K/F$ be an extension of fields, and let $W\subset K^N$ be a set
of vectors. We write
\[
  \Perp_F(W) = \{\bfb\in F^N : \bfb\cdot\bfw=0~\text{for all}~\bfw\in W \}
\]
for the subspace of~$F^N$ that is orthogonal to~$W$.  When
$W=\{\bfw\}$ consists of a single vector, we write
\[
  \Perp_F(\bfw) = \{\bfb\in F^N : \bfb\cdot\bfw=0 \}
\]
for $\Perp_F\bigl(\{\bfw\}\bigr)$.  Assuming that $\bfw\ne\bfzero$, we
note that $\Perp_K(\bfw)$ is simply a hyperplane in~$K^N$, but that in
general the~$F$-vector space $\Perp_F(\bfw)$ may have dimension
anywhere from~$0$ to~$N-1$.
\end{definition}

The following elementary facts will be useful.

\begin{lemma}
\label{lemma:perpKVFK}
Let $K/F$ be an extension of fields.
\begin{parts}
\Part{(a)}
Let $U\subset F^N$ be an $F$-vector subspace. Then
\[
  \Perp_K(U\otimes_F K) = \Perp_F(U)\otimes_F K.
\]
\Part{(b)}
Assume that $K/F$ is Galois, and let $V\subset K^N$ be a
$K$-vector subspace that is~$\Gal(K/F)$-invariant. 
Then $Y=V\cap F^N$ is the unique $F$-vector subspace of~$F^N$
satisfying $V=Y\otimes_FK$.
\Part{(c)}
Let $V\subset F^N$ be an $F$-vector subspace. Then
\[
  \Perp_F\bigl(\Perp_F(V)\bigr) = V.
\]
\Part{(d)}
Let $V_1,\ldots,V_t\subset F^N$ be $F$-vector subspaces. Then
\[
  \Perp_F(V_1) + \dots + \Perp_F(V_t)
  = \Perp_F(V_1\cap\cdots\cap V_t).
\]
\end{parts}
\end{lemma}
\begin{proof}
(a),~(c), and~(d) are linear algebra exercises, while~(b) is standard
linear algebra and Galois theory;
cf.\ \cite[Lemma~II.5.8.1]{MR2514094}.\EndNote{appendix2}
\end{proof}

The proof of Theorem~\ref{theorem:linrelsonlogP} uses the following
(qualitative) version of Baker's theorem on linear forms in
logarithms.

\begin{theorem}
\label{theorem:baker}
\textup{(Baker's Theorem)}
Let $\a_1,\ldots,\a_n\in\Qbar^*$ be algebraic numbers, and let
\[
  \bfw = \bigl(\log(\a_1),\ldots,\log(\a_n)\bigr)\in\CC^n.
\]
Then
\begin{equation}
  \label{eqn:perpqbareqperpq}
  \Perp_\Qbar(\bfw) \cong \Perp_\QQ(\bfw)\otimes_{\QQ} \Qbar.
\end{equation}
In particular, the $\Qbar$-vector space $\Perp_\Qbar(\bfw)\subset\Qbar^N$
is~$\Gal(\Qbar/\QQ)$-in\-var\-iant.
\end{theorem}
\begin{proof}
If $\dim\Perp_\Qbar(\bfw)=0$ or~$1$, then~\eqref{eqn:perpqbareqperpq}
says that
\[
  \log(\a_1),\ldots,\log(\a_n)
\]
are linearly dependent over~$\Qbar$ if and only if they are linearly
dependent over~$\QQ$, which is the usual statement of Baker's theorem;
see~\cite{MR1074572}.  The general case of~\eqref{eqn:perpqbareqperpq}
is then an easy induction on the dimension
of $\Perp_\Qbar(\bfw)$.\EndNote{appendix9}
Finally, the equality~\eqref{eqn:perpqbareqperpq} shows that
$\Perp_\Qbar(\bfw)$ is a~$\Gal(\Qbar/\QQ)$-invariant subspace
of~$\Qbar^N$.
\end{proof}

\begin{proof}[Proof of Theorem $\ref{theorem:linrelsonlogP}$]
To ease notation, we let $\rho=\rho(A)$ and $\ell=\ell(A)=\ell_\f$.
We also note from Proposition~\ref{proposition:monmapprops}(c) that
$\rho=\d_{\f_A}$, so~$\rho>1$. We take~$K/\QQ$ to be a finite Galois
extension containing the coordinates of~$P$ and the eigenvalues
of~$A$.
\par
Writing~$A$ in Jordan normal form, it is easy to see that the
matrices 
\begin{equation}
  \label{eqn:nellrhonAn}
  \bigl\{n^{-\ell}\rho^{-n}A^n : n\ge0\bigr\}
\end{equation}
lie in a bounded subset of~$\Mat_N(\CC)\cong\CC^{N^2}$, so by
compactness, any infinite subsequence of matrices
in~\eqref{eqn:nellrhonAn} has an accumulation point in~$\Mat_N(\CC)$.
\par
For $Q=(y_1,\ldots,y_N)\in K^N$ and $v\in M_K$, we let
\[
  \log\|Q\|_v = \text{the column vector
           ${}^t\left(\log\|y_1\|_v,\ldots,\log\|y_N\|_v\right)$}.
\]
This notation gives the convenient formula
\begin{equation}
  \label{eqn:logfAnPAnP}
  \log\bigl\|\f_A^n(P)\bigr\|_v = A^n\log\|P\|_v.
\end{equation}
Further, for any real vector~$\bfu=(u_1,\ldots,u_N)\in\RR^N$, we let
\[
  \max(\bfu) = \max\{u_1,\ldots,u_N\}
  \quad\text{and}\quad
  \maxplus(\bfu) = \max\{0,u_1,\ldots,u_N\}.
\]
\par
Let $P\in\GG_m^N(\Qbar)$ be a point satisfying~$\hhat_\f(P)=0$, and
let \text{$v\in M_K$}. Our first goal is to show that
$\Perp_\CC\bigl(\log\|P\|_v\bigr)$ contains a non-trivial $\CC$-vector
subspace of~$\CC^N$ that does not depend on~$P$ or~$v$.  Using the
definition of canonical height, we have
\[
  0 = \hhat_\f(P) =
  \limsup_{n\to\infty} \frac{h\bigl(\f_A^n(P)\bigr)}{n^\ell\rho^n}.
\]
We choose an infinite sequence of natural numbers~$\Ncal$
so that the limsup is a limit, i.e., 
\begin{equation}
  \label{eqn:0hfPlim}
  0 = \hhat_\f(P) =
  \lim_{n\in\Ncal} \frac{h\bigl(\f_A^n(P)\bigr)}{n^\ell\rho^n}.
\end{equation}
Replacing~$\Ncal$ with an infinite subsequence, which 
we again denote by~$\Ncal$, we may assume that the limit
\begin{equation} 
  \label{eqn:BeqlimAn}
  B = \lim_{n\in\Ncal} \frac{A^n}{n^\ell \rho^n} 
\end{equation}
also exists. We also note that since~$A$ has integer coefficients, the
coefficients of the matrix~$B$ are real, i.e., $B\in\Mat_N(\RR)$.
This is important because we are about to write down inequalities that
involve the coefficients of~$B$. On the other hand, it need not
be true that the coefficients of~$B$ are algebraic numbers, nor is~$B$
necessarily invertible.
\par
Writing the height as a sum over the places of~$K$, we have
\begin{align*}
  0 
  &= \lim_{n\in\Ncal}  \frac{1}{n^\ell\rho^n}
          \sum_{v\in M_K} \maxplus \left(\log\bigl\|\f_A^n(P)\bigr\|_v \right)
    \quad\text{from \eqref{eqn:0hfPlim},}\\
  &= \lim_{n\in\Ncal}  \frac{1}{n^\ell\rho^n} \sum_{v\in M_K} 
      \maxplus\bigl( A^n\log\|P\|_v \bigr)
    \quad\text{from \eqref{eqn:logfAnPAnP},}\\
  &= \sum_{v\in M_K} \maxplus 
    \left(\lim_{n\in\Ncal} \frac{A^n}{n^\ell\rho^n}\log\|P\|_v\right)\\
  &= \sum_{v\in M_K} \maxplus \bigl(B\log\|P\|_v\bigr)
    \quad\text{from \eqref{eqn:BeqlimAn}.}
\end{align*}
Since this sum of non-negative terms is equal to~$0$, we see that each
individual term must be equal to~$0$. We have thus proven that
\begin{equation}
  \label{eqn:maxpluslogPv0}
  \maxplus \bigl(B\log\|P\|_v\bigr) = 0 \quad\text{for all $v\in M_K$,}
\end{equation}
where~$B=\lim_{n\in\Ncal} n^{-\ell}\rho^{-n}A^n\in\Mat_N(\RR)$ is given
by~\eqref{eqn:BeqlimAn}.
\par
The definition of~$\maxplus$ and~\eqref{eqn:maxpluslogPv0} imply 
that\EndNote{appendix3}
\begin{equation}
  \label{eqn:maxBlogPvle0}
  \max\bigl( B\log\|P\|_v\bigr) \le 0 \quad\text{for all $v\in M_K$.}
\end{equation}
However, summing over~$v\in M_K$ and using the product formula shows
that
\begin{equation}
  \label{eqn:sumvMKmaxBlogPvge0}
  \sum_{v\in M_K} \max \bigl(B\log\|P\|_v\bigr) \ge 0.
\end{equation}
Thus the sum~\eqref{eqn:sumvMKmaxBlogPvge0} is non-negative,
but~\eqref{eqn:maxBlogPvle0} says that every term in the sum is
non-positive. It follow that every term in the
sum~\eqref{eqn:sumvMKmaxBlogPvge0} must vanish, which proves the
key formula
\begin{equation}
  \label{eqn:BlogPveq0}
  B\log\|P\|_v = 0 \quad\text{for all $v\in M_K$.}
\end{equation}
We note that~\eqref{eqn:BlogPveq0} is equivalent to
\begin{equation}
  \label{eqn:logPvinKerCB}
  \log\|P\|_v \in \Ker_\CC(B),
\end{equation}
where $\Ker_\CC(B)\subset\CC^N$ is independent of both~$P$ and~$v$.
\par
Formula~\eqref{eqn:BlogPveq0} says that the row vectors of~$B$
annihilate~$\log\|P\|_v$, but unfortunately the coordinates of~$B$ are
in~$\CC$, and our ultimate goal is to find vectors with
\emph{integer} coordinates that annihilate~$\log\|P\|_v$. To do this, we
study~$B$ and its kernel more closely.
\par  
Let~$V\subset\Qbar^N$ be a Jordan subspce for~$A$, and let $t=\dim
V$. In other words,~$V$ is an~$A$-invariant subspace of~$\Qbar^N$ and
there is a $\Qbar$-basis~$\Vcal$ for~$V$ so that the matrix of~$A|_V$
relative to the basis~$\Vcal$ is
\[
  \left[A|_V\right]_\Vcal 
  = \begin{pmatrix}
    \l & 1 & 0 & \cdots & 0 \\
     0 &\l & 1 & \cdots & 0 \\
     \vdots & & \ddots & \ddots & \vdots \\
     0 & 0 &  & \cdots & \l \\
  \end{pmatrix}
\]
We consider the limiting action of~$n^{-\ell}\rho^{-n}A^n$
on~$V$. There are two cases.
\par
First, if~$V$ is not a maximal Jordan subspace, then
either~$|\l|<\rho$ or~$t\le\ell$, so
\[
  \text{magnitude of largest entry of $\bigl[A^n|_V\bigr]_\Vcal$}
  \le O\bigl(n^{t-1}|\l|^n\bigr) = o(n^{\ell}\rho^n).
\]
Hence if~$V$ is not a maximal Jordan subspace, then
\begin{equation}
  \label{eqn:BonnonmaxJ}
  \lim_{n\to\infty} n^{-\ell}\rho^{-n}A^n|_V = 0.
\end{equation}
\par
Second, suppose that~$V$ is a maximal Jordan subspace, so$|\l|=\rho$
and~$t=\ell+1$.  Let $\Vcal=\{\bfv_1,\ldots,\bfv_t\}$ be the basis
of~$V$ used to put~$\bigl[A^n|_V\bigr]_\Vcal$ into Jordan normal
form, and let
\[
  W = \Qbar\bfv_1 + \cdots + \Qbar\bfv_{t-1} \subset V
\]
be the codimension~$1$ subspace of~$V$ generated by the first~$t-1$
vectors in the basis. Alternatively,
\begin{equation}
  \label{eqn:WeqkerAlt1V}
  W = \Ker \bigl( (A-\l)^{t-1}|_V \bigr) \subset V.
\end{equation}
Then~$W$ is~$A$-invariant, and the magnitude of the largest entry
of the matrix of~$A^n|_W$ relative to the
basis~$\{\bfv_1,\ldots,\bfv_{t-1}\}$ is
\[
  O\bigl(n^{t-2}|\l|^n\bigr)=O(n^{\ell-1}\rho^n)=o(n^\ell\rho^n),
\]
so we find that
\[
  \lim_{n\to\infty} n^{-\ell}\rho^{-n}A^n|_W = 0.
\]
On the other hand, the action of $A^n$ on the generator~$\bfv_t$
of~$V/W$ is given by
\[
  A^n\bfv_t = \l^n\bfv_t + \binom{n}{1}\l^{n-1}\bfv_{t-1}
    + \binom{n}{2}\l^{n-2}\bfv_{t-2}
    + \cdots  + \binom{n}{t-1}\l^{n-(t-1)}\bfv_1.
\]
The final term grows fastest at~$n\to\infty$,
so using the assumption that $t=\ell+1$, we find that
\[
  n^{-\ell}\rho^{-n}A^n\bfv_t
  = \frac{1}{\l^\ell\ell!} \left(\frac{\l}{\rho}\right)^n\bfv_1 + O(n^{-1}).
\]
We recall that~$\Ncal\subset\NN$ is a sequence such that
$n^{-\ell}\rho^{-n}A^n$ converges, so using the fact that~$\l/\rho$
has magnitude~$1$, we see that
\[
  \lim_{n\in\Ncal}   n^{-\ell}\rho^{-n}A^n\bfv_t
  = \frac{\xi}{\rho^\ell\ell!}\bfv_1
  \quad\text{for some $\xi\in\CC$ with $|\xi|=1$.}
\]
In particular, the limit is not~$0$.  Thus the action of~$B$
of~$V\otimes_\Qbar\CC$ satisfies
\[
  B|_{W\otimes_\Qbar\CC} = 0
  \quad\text{and}\quad
  B|_{(V/W)\otimes_\Qbar\CC} \ne 0.
\]
Hence
\begin{equation}
  \label{eqn:KerBVW}
  \Ker(B|_{V\otimes_\Qbar\CC}) = W\otimes_\Qbar\CC,
\end{equation}
where~$W\otimes_\Qbar\CC$ is an~$A$-invariant codimension~$1$ subspace
of~$V\otimes_\Qbar\CC$.
\par
We now write~$\Qbar^N$ as an (internal) direct sum of $A$-invariant
subspaces
\[
  \Qbar^N = V_1 \dotplus V_2 \dotplus\cdots\dotplus V_r \dotplus Z,
\]
where~$V_1,\ldots,V_r$ are the distinct maximal Jordan subspaces
for~$A$ and where~$Z$ is the direct sum of all of the other Jordan
subspaces for~$A$.  By definition, we have $r=r(A)$. Further, for
each~$i$ we let~$W_i\subset V_i$ be the~$A$-invariant codimension~$1$
$\Qbar$-subspace of~$V_i$ satisfying
\[
  \Ker(B|_{V_i\otimes_\Qbar\CC}) = W_i\otimes_\Qbar\CC
\]
as described in~\eqref{eqn:KerBVW}. Since we also have
$\Ker(B|_{Z\otimes_\Qbar\CC})=0$ from~\eqref{eqn:BonnonmaxJ}, we see
that the kernel of~$B$ acting on~$\CC^N\cong\Qbar^N\otimes_\Qbar\CC$ is
\begin{align*}
  \Ker_\CC(B) &= (W_1\otimes_\Qbar\CC) 
     \dotplus\cdots\dotplus(W_r\otimes_\Qbar\CC) 
       \dotplus(Z\otimes_\Qbar\CC) \\
  &= (W_1\dotplus\cdots\dotplus W_r\dotplus Z) \otimes_\Qbar\CC.
\end{align*}
To ease notation, we let
\begin{equation}
  \label{eqn:UeqW1WNZ}
  U = W_1\dotplus\cdots\dotplus W_r\dotplus Z \subset \Qbar^N,
\end{equation}
so~$U$ is a~$\Qbar$-vector space satisfying
\begin{equation}
  \label{eqn:KerCBeqWiZ}
  \Ker_\CC(B) = U\otimes_\Qbar\CC
  \quad\text{and}\quad
  \dim_\Qbar U = N-r.
\end{equation}
The dimension of~$U$ follows from the fact that each~$W_i$ has
codimension~$1$ in~$V_i$, and $V_1\dotplus\cdots\dotplus V_r\dotplus
Z=\Qbar^N$. We also note that~$U$ depends only on the matrix~$A$,
which is clear from above, or from the alternative description
of~$U$ as the kernel of the linear transformation\EndNote{appendix1}
\[
  \prod_{|\l|<\rho} (A-\l)^N \cdot \prod_{|\l|=\rho} (A-\l)^\ell
  \in\Mat_N(\Qbar).
\]
\par
Returning to our point~$P\in\GG_m^N(\Qbar)$ with height
$\hhat_{\f_A}(P)=0$, we next observe that
\begin{align}
  \label{eqn:perplogPsupperpU}
  \Perp_\CC\bigl(\log\|P\|_v\bigr)
  &\supset \Perp_\CC\bigl(\Ker_\CC(B)\bigr)
  &&\text{from \eqref{eqn:logPvinKerCB},} \notag\\
  &= \Perp_\CC(U\otimes_\Qbar\CC)
  &&\text{from \eqref{eqn:KerCBeqWiZ},} \notag\\
  &= \Perp_\Qbar(U)\otimes_\Qbar\CC
  &&\text{from Lemma \ref{lemma:perpKVFK}(a),} \notag\\
  &\supset \Perp_\Qbar(U).
\end{align}
\par
This gives us some vectors in~$\Qbar^N$ that annihlate~$\log\|P\|_v$,
but our goal is to find vectors in~$\QQ^N$ with this property.  We
note that the coordinates of the vector~$\log\|P\|_v$ are logarithms
of algebraic numbers, so we can apply Baker's theorem
(Theorem~\ref{theorem:baker}) to conclude that the $\Qbar$-vector space
$\Perp_\Qbar\bigl(\log\|P\|_v\bigr)$ is Galois invariant as a subspace
of~$\Qbar^N$.\footnote{%
    If $v$ is non-archimedean, then
    $\|x_i\|_v=p_v^{r_{i,v}}$ for a rational prime~$p_v$ and rational
    numbers~$r_{i,v}\in\QQ$, so~$\log\|P\|_v$ is a vector in~$\QQ^N$
    multiplied by the scalar~$\log p_v$. Thus in the non-archimedean case,
    the equality $\Perp_{\Qbar}\bigl(\log\|P\|_v\bigr)=
    \Perp_{\QQ}\bigl(\log\|P\|_v\bigr)\otimes_{\QQ}\Qbar$ is a triviality.
    But for archimedean~$v$, we appear to need the full strength of
    Baker's theorem.}
Hence~\eqref{eqn:perplogPsupperpU} implies that
\begin{equation}
  \label{eqn:perplogPsupGsum}
  \Perp_\Qbar\bigl(\log\|P\|_v\bigr) 
  \supset \sum_{\s\in\Gal(\Qbar/\QQ)}\s\bigl(\Perp_\Qbar(U)\bigr).
\end{equation}
The $\Qbar$-vector space on the right-hand side
of~\eqref{eqn:perplogPsupGsum} is~$\Gal(\Qbar/\QQ)$-in\-var\-iant, so
Lemma~\ref{lemma:perpKVFK}(b) says that there is a
(unique)~$\QQ$-vector space~$Y\subset\QQ^N$ such that
\begin{equation}
  \label{eqn:sumsPerpUY}
  \sum_{\s\in\Gal(\Qbar/\QQ)}\s\bigl(\Perp_\Qbar(U)\bigr) = Y \otimes_\QQ\Qbar.
\end{equation}
We stress here that~$Y$ depends only on~$A$ and is independent of~$P$
and~$v$, since the same is true of~$U$. It follows
from~\eqref{eqn:perplogPsupGsum} that
\begin{equation}
  \label{eqn:PerpQPsupY}
  \Perp_\QQ\bigl(\log\|P\|_v\bigr) \supset Y.
\end{equation}
We defer the computation of~$\dim_\QQ(Y)$,
which turns out to equal~$\rbar$,  until the end of the proof.
\par
We let
\[
  L = Y\cap\ZZ^N\subset\ZZ^N,
\]
so~$L$ is an integral lattice satisfying
\[
  \rank_\ZZ L =\dim_\QQ Y.
\]
It follows from~\eqref{eqn:PerpQPsupY} that
\[
  \bfe\cdot\log\|P\|_v=0
  \quad\text{for all $\bfe\in L$.}
\]
Writing
\[
  P = (x_1,\ldots,x_N) \in \GG_m^N(K)
  \quad\text{and}\quad
  \bfe=(e_1,\ldots,e_N)\in L,
\]
this becomes
\begin{equation}
  \label{eqn:prodxjejeqv1}
  \prod_{j=1}^N \|x_j\|_v^{e_j} =
  \biggl\|\prod_{j=1}^N x_j^{e_j}\biggr\|_v = 1,
\end{equation}
where note that we are allowed to move the~$e_j$ across the absolute
value signs because they are integers. Formula~\eqref{eqn:prodxjejeqv1}
holds for all~$v\in M_K$, so Kronecker's
theorem~\cite[Theorem~3.8]{MR2316407} implies that
\begin{equation}
  \label{eqn:prodxjejroot}
  \prod_{j=1}^N x_j^{e_j} \quad\text{is a root of unity
    for all $\bfe\in L$.}
\end{equation}
We stress that~\eqref{eqn:prodxjejroot} holds for all
points~$P\in\GG_m^N(\Qbar)$ satsifying~$\hhat_{\f_A}(P)=0$ and
for all $\bfe\in L$, where~$L$ is independent of~$P$.
\par
The lattice~$L$ is associated to an algebraic subgroup~$G_L$
of~$\GG_m^N$ in the usual way,
\begin{equation}
  \label{eqn:GLeqintL}
  G_L = \bigcap_{\bfe\in L} \{ X_1^{e_1}\cdots X_N^{e_N} = 1 \bigr\},
\end{equation}
and the dimension of~$G_L$ is given by
\begin{equation}
  \label{eqn:dimGLleNr}
  \dim G_L = N - \rank_\ZZ L = N - \dim_\QQ Y.
\end{equation}
Further, we see from~\eqref{eqn:prodxjejroot} that
if~$\hhat_{\f_A}(P)=0$, then some power of the coordinates of~$P$
gives a point in~$G_L(\Qbar)$, which shows that
\begin{equation}
  \label{eqn:Pht0inGLdiv}
  \bigl\{P\in\GG_m^N(\Qbar) : \hhat_{\f_A}(P)=0 \bigr\}
  \subset G_L(\Qbar)^\div.
\end{equation}
\par
It remains to prove that $\dim G_L=N-\rbar$, which
from~\eqref{eqn:dimGLleNr}, is equivalent to showing that
$\dim_\QQ{Y}=\rbar$. We compute
\begin{align*}
  \dim_\QQ Y
  &= \dim_\Qbar\biggl( \sum_{\s\in\Gal(\Qbar/\QQ)}\s\bigl(\Perp_\Qbar(U)\bigr)
     \biggr)
  \qquad\text{from \eqref{eqn:sumsPerpUY},} \\
  &= \dim_\Qbar\biggl( \sum_{\s\in\Gal(\Qbar/\QQ)}\Perp_\Qbar(\s U)
     \biggr)\\
  &= \dim_\Qbar \Perp_\Qbar\biggl( 
    \bigcap_{\s\in\Gal(\Qbar/\QQ)} \s U \biggr)
  \qquad\text{from Lemma \ref{lemma:perpKVFK}(d),}\hidewidth \\
  &= N-\dim_\Qbar \biggl(
    \bigcap_{\s\in\Gal(\Qbar/\QQ)} \s U \biggr) \\
  &= N-\dim_\Qbar \left(\bigcap_{\s\in\Gal(\Qbar/\QQ)} 
      \left(\sum_{i=1}^r \s W_i \dotplus \s Z  \right)\right)
  \qquad\text{from \eqref{eqn:UeqW1WNZ}.}
\end{align*}
Each~$W_i$ is a codimension one subspace of a maximal Jordan
block~$V_i$, so each distinct Galois conjugate of a maximal Jordan
block contributes codimension one to the intersection.  By definition,
the number of such conjugates is~$\rbar$, so the codimension of the
intersection is~$\rbar$, and hence the dimension of the
intersection is is~$N-\rbar$.  This proves that $\dim_\QQ Y=\rbar$,
which completes the proof of Theorem~\ref{theorem:linrelsonlogP}.
\end{proof}

\section{Proof of Corollaries of Theorem~\ref{theorem:linrelsonlogP}}
\label{section:proofofcorollaries}

In this section we give the proofs of Corollaries
\ref{corollary:ht0orbitnotdense}, \ref{corollary:irredcharpoly},
\ref{corollary:algdegformonmap}, and~\ref{corollary:hgt0iffaeqd} to
Theorem~\ref{theorem:linrelsonlogP}.

\begin{proof}[Proof of Corollary $\ref{corollary:ht0orbitnotdense}$]
For~$Q=(y_1,\ldots,y_N)\in\GG_m^N(\Qbar)$ and~$d\in\ZZ$, we
use the notation
\[
  Q^d = (y_1^d,\ldots,y_N^d).
\]
We let $G\subsetneq\GG_m^N$ be the algebraic subgroup described in the
statement of Theorem~\ref{theorem:linrelsonlogP} for the monomial
map~$\f$, and we let~$L\subset\ZZ^N$ be the lattice associated to~$G$
via~\eqref{eqn:GLeqintL}.
\par
The assumption that~$\hhat_{\f}(P)=0$ implies that
\[
  \hhat_{\f}\bigl(\f^n(P)\bigr) = \d_\f^n \hhat_{\f}(P)=0
  \quad\text{for all $n\ge0$,}
\]
so Theorem~\ref{theorem:linrelsonlogP} says that $\f^n(P)\subset
G(\Qbar)^\div$ for all~$n\ge0$.  However, the points~$\f^n(P)$ in the
orbit~$\Orbit_\f(P)$ are all defined over the number field~$K=\QQ(P)$,
so
\begin{equation}
  \label{eqn:OfPGQdivK}
  \Orbit_{\f}(P) \subset G(\Qbar)^\div \cap \GG_m^N(K).
\end{equation}
\par
Let $d$ be the number of roots of unity in~$K$. We are going to prove
that
\begin{equation}
  \label{eqn:OfPdinGK}
  \Orbit_{\f}(P)^d \subset G(K).
\end{equation}
This will complete the proof, since writing~$G=G_L$ to indicate the
dependence of~$G$ on the lattice~$L\subset\ZZ^N$, we clearly have
\[
  \{Q\in\GG_m^N(\Qbar): Q^d\in G_L(\Qbar)\bigr\} = G_{dL}(\Qbar).
\]
So~\eqref{eqn:OfPdinGK} implies that~$\Orbit_\f(P)\subset G_{dL}(K)$.
But~$G_{dL}$ is an algebraic subgroup of~$\GG_m^N$ of the same
dimension as~$G_L$, and hence~$\Orbit_\f(P)$ is contained in a proper
algebraic subgroup of~$\GG_m^N$.
\par
To prove the claim, let
\[
  Q = (y_1,\ldots,y_N) \in G(\Qbar)^\div \cap \GG_m^N(K).
\]
The assumption that $Q\in G(\Qbar)^\div$ means that there is an
$m\ge1$ such that $Q^m\in G(\Qbar)$; we take the smallest
such~$m$. The group~$G$ is defined by the lattice~$L$, so 
\[
  y_1^{e_1m}\cdots y_N^{e_Nm}=1
  \quad\text{for all $\bfe=(e_1,\ldots,e_N)\in L$.}
\]
Taking roots, this implies that
\[
  \text{$y_1^{e_1}\cdots y_N^{e_N}$ is an $m^{\text{th}}$-root of unity
      for all $\bfe\in L$.}
\]
But $Q$ is in~$\GG_m^N(K)$, so $y_1^{e_1}\cdots y_N^{e_N}\in K$, and
hence $y_1^{e_1}\cdots y_N^{e_N}$ is a $d^{\text{th}}$-root of unity.
It follows that~$m\mid d$, so in particular~$Q^d\in G(\Qbar)$.  But
also $Q\in\GG_m^N(K)$, so $Q^d\in G(K)$. This is true for every $Q \in
G(\Qbar)^\div \cap \GG_m^N(K)$, so~\eqref{eqn:OfPGQdivK} implies
that~$Q^d\in G(K)$ for every~$Q\in\Orbit_\f(P)$.
Hence~$\Orbit_\f(P)^d\in G(K)$, which completes the proof
of~\eqref{eqn:OfPdinGK}, and with it the proof of
Corollary~\ref{corollary:ht0orbitnotdense}.
\end{proof}

\begin{proof}[Proof of Corollary $\ref{corollary:irredcharpoly}$]
We already know that
\[
  P\in\PrePer(\f) \quad\Longrightarrow\quad  \hhat_\f(P)=0
\]
from Proposition~\ref{proposition:canhtprops}(c), so we assume that
$\hhat_{\f}(P)=0$, and we want to prove that~$P$ is preperiodic
for~$\f$.  
\par
We have assumed that $\det(T-A)$ is irreducible, so the
eigenvalues~$\l_1,\ldots,\l_N$ of~$A$ are distinct
and form a complete set of Galois
conjugates. Hence the Jordan subspaces are all $1$-dimensional
and are pairwise Galois conjugate. Since at least one of them
is a maximal Jordan subspace, we see that $\rbar(A)=N$, and
hence the algebraic subgroup~$G\subset\GG_m^N$ described
in Theorem~\ref{theorem:linrelsonlogP} has dimension~$0$.
It follows that the divisible hull of~$G$ is given by
\[
  G(\Qbar)^\div = \GG_m^N(\Qbar)_\tors,
\]
where the torsion subgroup~$\GG_m^N(\Qbar)_\tors$ consists of all
points whose coordinates are roots of unity.
Theorem~\ref{theorem:linrelsonlogP} and the assumption
that $\hhat_\f(P)=0$ imply that $P\in G(\Qbar)^\div$. It is then clear
that~$P$ is preperiodic for the monomial map~$\f$, since 
the coordinates of~$\f^n(P)$ are all roots of unity lying in the
number field~$\QQ(P)$, so take on only finitely many possible values.
\end{proof}

\begin{proof}[Proof of Corollary $\ref{corollary:algdegformonmap}$]
If~$\d_\f=1$, then Proposition~\ref{proposition:hfnPledfnhP}
tells us that $1\le \a_\f(P)\le\d_\f=1$, so~$\a_\f(P)=1$ for all
points~$P$. We assume for the remainder of the proof that~$\d_\f>1$.
\par
The following fact will be useful later in the proof.  Letting
$P^d=(x_1^d,\ldots,x_N^d)$ and noting that~$\f$ is a homomorphism
of~$\GG_m^N$, we have
\[
  h\bigl(\f^n(P^d)\bigr)=h\bigl(\f^n(P)^d\bigr)=dh\bigl(\f^n(P)\bigr).
\]
It then follows directly from the definitions of canonical height and
arithmetic degree that
\begin{equation}
  \label{eqn:afPdeqafP}
  \hhat_\f(P^d)=d\hhat_\f(P)
  \quad\text{and}\quad
  \a_\f(P^d)=\a_\f(P)
  \quad\text{for any~$d\ge1$}.
\end{equation}
\par
Let $G\subset\GG_m^N$ be the smallest algebraic subgroup of~$\GG_m^N$
that contains the orbit~$\Orbit_\f(P)$. The group~$G$ might not be
connected, but its identity component~$G_0$ has finite index in~$G$,
say~$d=(G:G_0)$.  Then the orbit of~$P^d$ is contained in~$G_0$, so
the smallest algebraic subgroup of~$\GG_m^N$
containing~$\Orbit_\f(P^d)$ is~$G_0$.  Since~$\a_\f(P^d)=\a_\f(P)$
from~\eqref{eqn:afPdeqafP}, we may replace~$P$ with~$P^d$, which
reduces us to the case that the group~$G$ is connected.
\par
We next note that since~$\f$ is a homomorphism of~$\GG_m^N$, the
smallest algebraic subgroup~$G$ containing the orbit~$\Orbit_\f(P)$ is
itself~$\f$ invariant, i.e.,~$\f(G)\subset G$.\EndNote{appendix6} The
group~$G$ corresponds to a lattice $L\subset\ZZ^N$. More precisely,
there is a perfect pairing
\[
  \ZZ^N\times\GG_m^N\longrightarrow\GG_m,\quad
  (\bfe,\bfx)\longmapsto \bfx^\bfe = \prod_{i=1}^N x_i^{e_i},
\]
and~$G_L$ is by definition the right kernel of~$L$ for this pairing.
The inclusion $\f(G_L)\subset G_L$ is equivalent to the inclusion
$LA\subset A$, where~$A$ is the matrix associated to the monomial
map~$\f$, and the restriction of~$\f$ to~$G_L$ is a monomial map whose
associated linear transformation (over~$\QQ$) is the restriction
of~$A$ to~$L^\perp$, where
\[
  L^\perp=\{\bff\in\ZZ^N : \text{$\bff\cdot\bfe=0$ for all $\bfe\in L$}\}.
\]
More precisely, there is a finite homomorphism~$\GG_m^{N-k}\to G_L$
such that~$\f$ induces a monomial map~$\GG_m^{N-k}\to\GG_m^{N-k}$
whose associated linear transformation (extended to~$\QQ$) is
isomorphic to the restriction of~$A$ to~$\L^\perp\otimes\QQ$.  We
write $\f_L$ for the induced monomial map on $\GG_m^{N-k}$, we let
$A_L^\perp$ be the linear transformation that~$A$ induces
on~$L^\perp\otimes\QQ$, and we replace~$P$ by a point in~$\GG_m^{N-k}$
that maps to~$P$.
\par
We apply Corollary~\ref{corollary:ht0orbitnotdense} to
the~$\f|_L$-orbit of~$P$. We chose~$G_L$ to be the smallest algebraic
subgroup of~$\GG_m^N$ that contains~$\Orbit_\f(P)$, so
Corollary~\ref{corollary:ht0orbitnotdense} tells us
that~$\hhat_{\f_L,G_L}(P)>0$.  On the other hand, we have the
implications
\begin{align}
  \label{eqn:afLGLPeqrALperp}
  \hhat_{\f_L,G_L}(P)>0
  &\;\Longrightarrow\;
  \a_{\f_L,G_L}(P) = \d_{\f_L,G_L}
  &&\text{from Proposition \ref{proposition:canhtprops}(d),} \notag\\*
  &\;\Longrightarrow\;
  \a_{\f_L,G_L}(P) = \rho(A_L^\perp)
  &&\text{from Proposition \ref{proposition:monmapprops}(c).} 
\end{align}
The characteristic polynomial of~$A_L^\perp$ as a linear transormation
of~$L^\perp\otimes\QQ$ is a polynomial in~$\QQ[T]$ that divides the
characteristic polynomial of~$A$ as a linear transformation
of~$\QQ^N$, and hence~$\rho(A_L^\perp)$ is the largest root of some
factor of~$\det(T-A)$ in~$\QQ[T]$. 
We also note that~$\a_{\f_L,G_L}(P)=\a_\f(P)$, since for any
morphism~$i:\PP^k\to\PP^N$ with  $\dim i(\PP^k)=k$,
we have $h_{\PP^N}\circ i\asymp h_{\PP^k}$,
so we can compute the arithmetic degree on either~$\GG_m^N$ 
or~$G_L$.\EndNote{appendix7} This proves that
\[
  \bigl\{\a_{\f}(P) : P \in\GG_m^N(\Qbar)\bigr\}
  \subset \bigl\{1,\rho(f_1),\rho(f_2),\ldots,\rho(f_s)\bigr\}.
\]
\par
For the opposite inclusion, let~$f(T)\in\ZZ[T]$ be a monic
irreducible factor of the characteristic polynomial of~$A$
with~$\rho(f)>1$, and write
\[
  \det(T-A) = f(T)^eg(T)\quad\text{with $f(T)\nmid g(T)$.}
\]
Evaluation at~$\f$ gives a map~$\ZZ[T]\to\End(\GG_m^N)$, i.e.,
\[
  \left(\sum b_iT^i\right) \cdot P = \prod \f^i(P)^{b_i}.
\]
In particular, note that~$\det(T-A)$ annihliates~$\GG_m^N$, since a
linear transformation is a root of its own characteristic polynomial.
Consider the subgroup of~$\GG_m^N$ defined by
\[
  G = g(T)\cdot\GG_m^N.
\]
Then~$G$ is a nontrivial~$\f$-invariant algebraic subgroup of~$\GG_m^N$
satisfying~$f(T)^e\cdot G=0$. The group~$G$ is a torus, and the monomial
map~$\f_G:G\to G$ satisfies $f(\f_G)^e=1$, so the associated
matrix~$A_G$ satisfies~$f(A_G)^e=1$.  Since~$f$ is irreducible, it
follows that the characteristic polynomial of~$A_G$ is a positive
power of~$f(T)$, and then Proposition~\ref{proposition:monmapprops}(c)
tells us that~$\d_{\f_G}=\rho(A_G)=\rho(f)$. Further, if we take any
$P\in G(\Qbar)$ whose orbit~$\Orbit_\f(P)$ is Zariski dense in~$G$,
then Corollary~\ref{corollary:ht0orbitnotdense} tells
us that~$\hhat_{G,\f|_G}(P)>0$, and then
Proposition~\ref{proposition:canhtprops}(d) says that
\[
  \a_\f(P) = \a_{G,\f|_G}(P) = \d_{\f_G} = \rho(f).
\]
Since we also always have $\a_\f(1)=1$,
this completes the proof of the other inclusion
\[
  \bigl\{\a_{\f}(P) : P \in\GG_m^N(\Qbar)\bigr\}
  \supset \bigl\{1,\rho(f_1),\rho(f_2),\ldots,\rho(f_s)\bigr\},
\]
and with it, the proof of Corollary~\ref{corollary:algdegformonmap}.
\end{proof}

\begin{proof}[Proof of Corollary $\ref{corollary:hgt0iffaeqd}$]
The implication $\hhat_{\f}(P)>0\Rightarrow\a_{\f}(P)=\d_{\f}$ is easy
to prove for all dominant rational maps; see
Proposition~\ref{proposition:canhtprops}(d).  So we need to prove the
opposite implication under the assumption that~$A$ is 
diagonalizable.\EndNote{appendix8} We note that the diagonalizability
condition implies in particular that~$\ell(A)=0$, since every Jordan
block has dimension~$1$.
\par
Let~$P$ be a point with~$\hhat_{\f}(P)=0$.  We continue with the
notation from the proof of Corollary~\ref{corollary:algdegformonmap},
so in particular~$G=G_L\subset\GG_m^N$ is the smallest algebraic
subgroup containing~$\Orbit_\f(P)$ and~$\f_L$ is the restriction
of~$\f$ to~$G_L$. We proved that
\begin{equation}
  \label{eqn:afPeqdfL}
  \a_\f(P) =  \a_{\f_L,G_L}(P) = \rho(A_L^\perp) = \d_{\f_L};
\end{equation}
see~\eqref{eqn:afLGLPeqrALperp} and the remark following for the first
two equalities, while the third equality is
Proposition~\ref{proposition:monmapprops}(c).  The assumption
that~$G_L$ is the smallest algebraic subgroup
containing~$\Orbit_\f(P)$, combined with
Corollary~\ref{corollary:ht0orbitnotdense}, implies that
$\hhat_{G_L,\f_L}(P)>0$. Using the definition of canonical height, we
have
\begin{align*}
  0 &= \hhat_{\f}(P) = \lim_{n\to\infty} \frac{1}{\d_\f^n}h\bigl(\f^n(P)\bigr),\\
  0 &< \hhat_{\f_L,G_L}(P) 
      = \lim_{n\to\infty} \frac{1}{\d_{\f_L}^n}h\bigl(\f^n(P)\bigr).
\end{align*}
It follows that $\d_{\f_L}<\d_\f$ (strict inequality),
which together with~\eqref{eqn:afPeqdfL} gives the desired
inequality $\a_\f(P)<\d_\f$.
\end{proof}

\section{Regular affine automorphisms}
\label{section:regaffaut}

A rational map is \emph{algebraically stable} if $\deg(\f^n)$ is equal
to~$(\deg\f)^n$ for all~$n\ge1$, which from
Proposition~\ref{proposition:df=inf} is equivalent to~$\d_\f=\deg\f$.
As noted in Section~\ref{subsection:regaffauts}, regular affine
automorphisms are algebraically stable. In this section we describe
Kawaguchi's theory of canonical heights and use it to illustrate some
of the ideas in this paper and to raise a question about birational
maps.

\begin{theorem}
\label{theorem:kawacanht}
Let $\f:\AA^N\to\AA^N$ be a regular affine automorphism, i.e., the
extensions of~$\f$ and~$\f^{-1}$ to~$\PP^N$ are not morphisms, but
satisfy $Z(\f)\cap Z(\f^{-1})=\emptyset$.  Let $\hhat_{\f}$ and
$\hhat_{\f^{-1}}$ be the canonical heights associated to~$\f$
and~$\f^{-1}$, respectively.
\begin{parts}
\Part{(a)}
$\d_\f=\deg(\f)$ and $\d_{\f^{-1}}=\deg(\f^{-1})$.
\Part{(b)}
$h(P)+O(1)\le\hhat_\f(P)+\hhat_{\f^{-1}}(P)\le2h(P)+O(1)$  
for all~$P\in\AA^N(\Qbar)$.
\Part{(c)}
$
  \hhat_{\f}(P)=0
  \quad\Longleftrightarrow\quad
  \hhat_{\f^{-1}}(P)=0
  \quad\Longleftrightarrow\quad
  P\in\Per(\f).
$
\end{parts}
\end{theorem} 
\begin{proof}
(a) The fact that regular affine automorphisms are algebraically stable
is well-known; see for example~\cite[Chapter~2]{sibony:panoramas}.
\par\noindent(b,c)
The construction
is due to Kawaguchi~\cite{arxiv0405007}
(see also~\cite[Exercises~7.17--7.22]{MR2316407}),
but at the time it was only known to work for~$N=2$.  The deep height
inequality needed to justify the construction for general~$N$ was
proven independently by Kawaguchi~\cite{arxiv0909.3573} and
Lee~\cite{arxiv0909.3107}.
\end{proof}

Using Theorem~\ref{theorem:kawacanht}, it is easy to
compute~$\a_\f(P)$ for regular affine automorphisms.

\begin{corollary}
\label{corollary:afPregaffauts}
Let $\f:\PP^N\dashrightarrow\PP^N$ be a regular affine automorphism
defined over~$\Qbar$, and let $P\in\AA^N(\Qbar)$. Then
\[
  \a_\f(P) =\begin{cases}
    \d_\f &\text{if $P\notin\Per(\f)$,} \\
     1 &\text{if $P\in\Per(\f)$.} \\
  \end{cases}
\]
\end{corollary}
\begin{proof}
If $P\in\Per(\f)$, it is clear that $\a_\f(P)=1$.  Conversely,
if~$P\notin\Per(\f)$, then Theorem~\ref{theorem:kawacanht}(c) says
that $\hhat_\f(P)>0$, so Proposition~\ref{proposition:canhtprops}
implies that $\a_\f(P)=\d_\f$.
\end{proof}

\begin{remark}
\label{remark:hhatasymph}
If $\f:\PP^N\dashrightarrow\PP^N$ is a dominant birational map with
$\d_\f>1$ and $\d_{\f^{-1}}>1$, then there are canonical heights
associated to both~$\f$ and~$\f^{-1}$.  Neither of these canonical
heights can individually satsify $\hhat\asymp h$, but Kawaguchi's
construction suggests looking at the sum $\hhat_\f + \hhat_{\f^{-1}}$.
It would be very interesting to give general conditons which imply
that $\hhat_\f + \hhat_{\f^{-1}}\asymp h$, since it is an exercise to
prove that\EndNote{appendix4}
\[
  \hhat_\f + \hhat_{\f^{-1}}\asymp h\quad\text{and}\quad\hhat_\f(P)=0
  \quad\Longrightarrow\quad P\in\Per(\f).
\]
\end{remark}

\section{Dominant Self-Maps of General Varieties}
\label{section:generalizations}
Up to now we have restricted attention to rational self-maps of
$\PP^N$. In this section we describe how
Conjecture~\ref{conjecture:setofafPs} may be extended to arbitrary
varieties.  As usual, for any endomorphism~$F:V\to V$ of a finite
dimensional $\CC$-vector space, we write
\[
  \rho(F)
  = \max\bigl\{|\l| : \text{$\l$ is an eigenvalue of $F$}\bigr\}
\]
for the spectral radius of~$F$.

\begin{definition}
Let $X$ be a nonsingular irreducible algebraic variety, and let
$\psi:X\dashrightarrow X$ be a dominant rational map. Then~$\psi$
induces a $\QQ$-linear endomorphism~$\psi^*$ of the rational
N\'eron--Severi group $\NS(X)_\QQ=\NS(X)\otimes\QQ$.
N.B. In general, $(\psi_1\circ\psi_2)^*\ne\psi_2^*\circ\psi_1^*$.
Let~$\f:X\dashrightarrow X$ be a dominant rational map.  The
(\emph{first}) \emph{dynamical degree of~$\f$} is
\[
  \d_\f = \lim_{n\to\infty} \rho\bigl((\f^n)^*\bigr)^{1/n}. 
\]
We note that if $X=\PP^N$, then $\NS(X)_\QQ=\QQ$ and
$\rho(\f^*)=\deg(\f)$, so this definition is consistant with our
earlier definition. 
\end{definition}

\begin{definition}
Let $X$ be a nonsingular irreducible algebraic variety defined
over~$\Qbar$, let $\f:X\dashrightarrow X$ be a dominant rational map
defined over~$\Qbar$, and fix a height function~$h_X$ on~$X(\Qbar)$
associated to an ample divisor. Also let
\[
  X(\Qbar)_\f = \bigl\{P\in X(\Qbar) : \Orbit_\f(P)\cap Z(\f)=\emptyset\bigr\}.
\]
Then for~$P\in X(\Qbar)_\f$, we define the \emph{arithmetic degree
  of~$\f$ at~$P$} to be
\[
  \a_\f(P) = \limsup_{n\to\infty} h_X\bigl(\f^n(P)\bigr)^{1/n}.
\]
\end{definition}

\begin{remark}
The definition of~$\a_\f(P)$ is independent of the choice of height
function on~$X$, because if~$D$ and~$E$ are ample divisors on~$X$, then
$h_{X,D}\asymp h_{X,E}$. Hence there is a constant~$C>0$ such that
\[
  C^{-1} h_{X,E}\bigl(\f^n(P)\bigr)-C
  \le h_{X,D}\bigl(\f^n(P)\bigr) \le
  C h_{X,E}\bigl(\f^n(P)\bigr)+C.
\]
Taking the $n^{\text{th}}$-root and letting $n\to\infty$ shows that~$h_{X,D}$
and~$h_{X,E}$ yield the same value of~$\a_\f(P)$.
\end{remark}

We now generalize Conjecture~\ref{conjecture:setofafPs}.

\begin{conjecture}
\label{conjecture:afPdfonvarieties}
Let $X$ be a nonsingular irreducible algebraic variety defined
over~$\Qbar$, and let $\f:X\dashrightarrow X$ be a dominant rational
map defined over~$\Qbar$.
\begin{parts}
\Part{(a)}
The set
\[
  \bigl\{\a_\f(P) :  P\in X(\Qbar)_\f \bigr\}
\]
is a finite set of algebraic integers. 
\Part{(b)}
Let $P\in X(\Qbar)_\f$ be a point such that $\Orbit_\f(P)$ is
Zariski dense in~$X$. Then $\a_\f(P)=\d_\f$.
\end{parts}
\end{conjecture}

If~$\f$ is a morphism, then $\rho({\f^*}^n)=\rho(\f^*)^n$, so in
particular~$\d_\f=\rho(\f^*)$ is an algebraic integer. However, even
for morphisms, Conjecture~\ref{conjecture:afPdfonvarieties} appears to
be nontrivial in general. We now show that it is true for the K3
surfaces and automorphisms studied in~\cite{silverman:K3heights}.

\begin{theorem}
Let $X\subset\PP^2\times\PP^2$ be a smooth surface given by the
intersection of a~$(2,2)$-form and a~$(1,1)$-form, and assume that
$\NS(X)\cong\ZZ^2$.  The two projections~$\pi_1,\pi_2:X\to\PP^2$
induce noncommuting involutions~$\iota_1,\iota_2:X\to\PP^2$, and
the map~$\f=\iota_1\circ\iota_2$ is a automorphism of~$X$ of
infinite order. \textup(See~\cite{silverman:K3heights} for 
details.\textup) Then
$\d_\f=7+4\sqrt3$, 
\[
  \a_\f(P)=\d_\f
  \;\Longleftrightarrow\;
  P\notin\Per(\f)
  \;\Longleftrightarrow\;
  \text{$\Orbit_\f(P)$ is Zariski dense,}
\]
and 
\[
  \a_\f(P)=1
  \;\Longleftrightarrow\;
  P\in\Per(\f)
  \;\Longleftrightarrow\;
  \text{$\Orbit_\f(P)$ is not Zariski dense.}
\]
In particular, Conjecture~$\ref{conjecture:afPdfonvarieties}$ is true.
\end{theorem}
\begin{proof}
To ease notation, let $\b=7+4\sqrt3$.  It is shown
in~\cite{silverman:K3heights} that~$\f^*$ acts on the natural
basis~$\{\pi_1^*H,\pi_2^*H\}$ of~$\Pic(X)=\NS(X)$ via the matrix
$\left(\begin{smallmatrix} -1&4\\ -4&15\\ \end{smallmatrix}\right)$.
This matrix has eigenvalues~$\b$ and~$\b^{-1}$, and~$\f$
is a morphism, so $\d_\f=\b$.

It is further shown that there are
divisors~$E^+,E^-\in\NS(X)\otimes\RR$ satisfying $\f^*E^+=\b E^+$ and
$\f^*E^- = \b^{-1}E^-$ and such that~$E^++E^-$ is in the ample cone.
Writing the associated canonical height functions as~$\hhat^+$
and~$\hhat^-$, the function~$\hhat=\hhat^++\hhat^-$ is a Weil height
function associated to an ample divisor, so we can use it to
compute~$\a_\f(P)$. In particular, it is proven
in~\cite{silverman:K3heights} that
\begin{align*}
  \hhat^+(P)=0 
  \;\Longleftrightarrow\;
  \hhat^-(P)=0 
  &\;\Longleftrightarrow\;
  \hhat(P)=0
  \;\Longleftrightarrow\;
  P\in\Per(\f) \\
  &\;\Longleftrightarrow\;
  \text{$\Orbit_\f(P)$ is not Zariski dense}.
\end{align*}
\par
It is clear that if~$P\in\Per(\f)$, then~$\a_\f(P)=1$. Suppose now
that~$P\notin\Per(\f)$, or equivalently, that~$\Orbit_\f(P)$ is
Zariski dense. Then
\[
  \hhat\bigl(\f^n(P)\bigr)
  = \hhat^+\bigl(\f^n(P)\bigr) + \hhat^-\bigl(\f^n(P)\bigr)
  = \b^n\hhat^+(P)+\b^{-n}\hhat^-(P).
\]
Since~$\b>1$ and~$\hhat^+(P)>0$, taking $n^{\text{th}}$-roots and
letting~$n\to\infty$ gives $\a_\f(P)=\b=\d_\f$.
\end{proof}




\def\cprime{$'$}

\ifArXivVersion\else
\end{document}
\fi
\appendix 

\section{Additional material}
In this appendix we give further details and comments regarding
various statements in the body of the article. This appendix is for
the ArXiv version of this article; it will not appear in the published
version.
\subsection{Description of the divisible hull}
\label{appendix5}\hfill\break
Let~$G$ be an algebraic subgroup of $\GG_m^N$.  We verify that
$G(\Qbar)^\div=G(\Qbar)\GG_m^N(\Qbar)_\tors$.
\par
Let~$Q\in G(\Qbar)$ and $\bfzeta\in\GG_m^N(\Qbar)_\tors$.  Choose
an~$n\ge1$ such that $\bfzeta^n=1$.  Write $G=G_L$ for a lattice
$L\subset\ZZ^N$, i.e., with the obvious notation, the group~$G$ is the
set of points satisfying $P^\bfe$ for every~$\bfe\in L$.  Then
every~$\bfe\in L$ we have
\[
  \bigl((\bfzeta Q)^n\bigr)^\bfe = (\bfzeta^n)^\bfe (Q^\bfe)^n = 1\cdot 1 = 1,
\]
which proves that $\bfzeta Q\in G(\Qbar)^\div$.
\par
For the converse, we suppose that $P\in G_L(\CC)^\div$,
say $P^n\in G_L(\CC)$ for some $n\ge1$. Thus
$P^{n\bfe}=1$ for all $\bfe\in L$, so $P^\bfe\in\bfmu_n$ for all $\bfe\in L$.
In this way we get a homomorphism
\[
  \xi:L \longrightarrow \bfmu, \quad \xi_\bfe = P^\bfe.
\]
We want to prove that there is an element $\bfzeta\in\bfmu^N$ with the
property that $\bfzeta^\bfe=P^\bfe$ for all~$\bfe\in L$, since then
$\bfzeta^{-1}P\in G_L(\CC)$ and $\bfzeta\in\GG_m^N(\CC)_\tors$.
\par
Notationally it's easier if we identify $\bfu$ with $\QQ/\ZZ$
via the map $t\mapsto e^{2\pi it}$. Then~$\xi$ is a homomorphism
\[
  \xi \in \Hom(L,\QQ/\ZZ),
\]
and we want to know if $\xi$ lifts to an element of
$\Hom(\ZZ^N,\QQ/\ZZ)\cong\Hom(\QQ/\ZZ)^N$. In other words, we want to
know if the map
\[
  \Hom(\ZZ^N,\QQ/\ZZ)\longrightarrow \Hom(L,\QQ/\ZZ)
\]
induced by $L\subset\ZZ^N$ is surjective. Letting $K=\ZZ^N/L$, this is
equivalent to showing that $\Ext^1(K,\QQ/\ZZ)=0$. Since~$K$ is a
direct sum of cyclic groups~$C_m$ and copies of~$\ZZ$, it suffices to
prove that
\[  
  \Ext^1(C_m,\QQ/ZZ)=\Ext^1(\ZZ,\QQ/ZZ)=0.
\]
Applying~$\Hom(\,\cdot\,,\QQ/\ZZ)$ to the exact sequence
$0\to\ZZ\xrightarrow{m}\ZZ\to C_m\to0$, we see that
\[
  \Ext^1(C_m,\QQ/ZZ)=\Ext^1(\ZZ,\QQ/ZZ)[m],
\]
so we are reduced to proving that $\Ext^1(\ZZ,\QQ/ZZ)=0$.  But $\ZZ$
is projective, so $\Ext^1(\ZZ,A)=0$ for any abelian group.

\subsection{Proof of Lemma~\ref{lemma:perpKVFK}(d)}
\label{appendix2}\hfill\break
We first prove that
\begin{equation}
  \label{eqn:firstformula}
  \Perp_F(V_1 \dotplus \dots \dotplus V_t)
  = \Perp_F(V_1)\cap\cdots\cap\Perp_F(V_t).
\end{equation}
Let $\bfw\in\Perp_F(V_1 \dotplus \dots \dotplus V_t)$.  Then~$\bfw$
certainly annihilates every~$V_i$, so $\bfv\in\Perp(V_i)$ for all $i$,
and hence $\bfv\in\Perp_F(V_1)\cap\cdots\cap\Perp_F(V_t)$. This proves
that
\[
  \Perp_F(V_1 \dotplus \dots \dotplus V_t)
  \subset
  \Perp_F(V_1)\cap\cdots\cap\Perp_F(V_t).
\]
Next let $\bfu\in\Perp_F(V_1)\cap\cdots\cap\Perp_F(V_t)$,
and let $\bfv\in V_1 \dotplus \dots \dotplus V_t$. Then
$\bfv=\bfv_1+\cdots+\bfv_N$ with $\bfv_i\in V_i$. But~$\bfu\in\Perp_F(V_i)$
for all~$i$, so $\bfu\cdot\bfv_i=0$ for all~$i$, so $\bfu\cdot\bfv=0$.
Hence $\bfu\in\Perp_F(V_1 \dotplus \dots \dotplus V_t)$, which proves the
other inclusion
\[
  \Perp_F(V_1)\cap\cdots\cap\Perp_F(V_t)
  \subset
  \Perp_F(V_1 \dotplus \dots \dotplus V_t).
\]
This proves~\eqref{eqn:firstformula}.
\par
We next use~\eqref{eqn:firstformula}, replacing~$V_i$,
with~$\Perp_F(V_i)$, and use~(c) to delete double perps. This gives
\begin{align*}
  V_1\cap\cdots\cap V_t
  &= \Perp_F(\Perp_F(V_1))\cap\cdots\cap\Perp_F(\Perp_F(V_t))
  \quad\text{from (c),} \\
  &= \Perp_F(\Perp_F(V_1)\dotplus\dots\dotplus\Perp_F(V_t))
  \quad\text{from \eqref{eqn:firstformula}.}
\end{align*}
Applying~$\Perp_F$ to this equality and using~(c) again gives
\[
  \Perp_F(V_1\cap\cdots\cap V_t)
  = \Perp_F(V_1)\dotplus\dots\dotplus\Perp_F(V_t),
\]
which is the desired result.

\subsection{Baker's theorem}
\label{appendix9}\hfill\break
We prove that
\[
  \Perp_\Qbar(\bfw) \cong \Perp_\QQ(\bfw)\otimes_{\QQ} \Qbar
\]
by induction on the dimension~$k$ of $\Perp_\Qbar(\bfw)$.  The result
is trivial if~$k=0$, and as already noted, the case $k=1$ is the
classical statement of Baker's theorem. Assume now that we know the
result for~$k$, and let $\dim\Perp_\Qbar(\bfw)=k+1$.  Write generators
for the relations in $\Perp_\Qbar(\bfw)$ as the rows of a
$(k+1)$-by-$N$ matrix~$B$, so we have $B\bfw=0$. Permuting the rows
of~$B$ and the coordinates of~$\bfw$, we may assume that $b_{k+1,N}\ne0$,
and then subtracting multiples of the last row of~$B$ from the other
rows, we may assume that $b_{iN}=0$ for $1\le i\le k$. We now let
\[
  \bfw'=(w_1,\ldots,w_{N-1})
  \quad\text{and}\quad
  B' = (b_{ij})_{\substack{1\le i\le k\\ 1\le j\le N-1\\}}
\]
Then the rows of~$B'$ generate~$\Perp_\Qbar(\bfw')$, so by the
induction hypothesis, the space~$\Perp_\Qbar(\bfw')$ has a basis
in~$\Perp_\QQ(\bfw)$. This means that we can
replace~$B'\in\Mat_{k\times N}(\Qbar)$ with a matrix in~$\Mat_{k\times
  N}(\QQ)$, and hence we may assume that the first~$k$ rows of~$B$
have coefficients in~$\QQ$ (and the final entry in each of these rows
is~$0$). 
\par
We now repeat the argument with a different row and column.
The first row of~$B$ must have a non-zero entry (and note
that the last entry is zero), so relabeling the first~$N-1$
coordinates of~$\bfw$, we may assume that~$b_{11}\ne0$. Subtracting
multiples of the first row from the other rows, we may further
assume that $b_{i1}=0$ for all $2\le i\le k+1$. We let
\[
  \bfw''=(w_2,\ldots,w_N)
  \quad\text{and}\quad
  B'' = (b_{ij})_{\substack{2\le i\le k+1\\ 2\le j\le N\\}},
\]
so the rows of~$B''$ generate~$\Perp_{\Qbar}(\bfw'')$. Again by the
induction hypothesis, the space~$\Perp_{\Qbar}(\bfw'')$ has a basis
in~$\Perp_\QQ(\bfw)$, and since the last column of~$B''$ is not zero
(since $b_{k+1,N}\ne0$), there must be some
vector~$(c_2,\ldots,c_N)\in\Perp_\QQ(\bfw'')$ with $c_N\ne0$. This
vector is not in the~$\QQ$-span of the first~$k$ rows of~$B''$,
since the last coordinates of the first~$k$ rows of~$B''$ are all
zero. Hence the vector
\[
  \bfc = (0,c_2,\ldots,c_N)
\]
is in~$\Perp_\QQ(\bfw)$ and is not in the span of the first~$k$ rows
of~$B$. This proves that the first~$k$ rows of~$B$ and the
vector~$\bfc$ generate a $\QQ$-vector subspace of~$\Perp_\QQ(\bfw)$ of
dimension~$k+1$, which is equal to the dimension
of~$\Perp_\Qbar(\bfw)$.
\par
\emph{Amusing remark}: There are many contrived examples of incorrect
induction proofs in which the case $k=0$ is easy, and if $k\ge1$, then
the proof from $k$ to $k+1$ is easy, but one glosses over the fact
that the induction argument is incorrect when one tries to go from
$k=0$ to $k=1$. The above proof has this form, i.e., $k=0$ is easy,
and $k$ implies $k+1$ is easy for $k\ge1$. Of course, the full proof
is correct because $k=1$ is also true, but the $k=1$ case is not
provable by a trivial induction from the $k=0$ case. Indeed, as noted,
the case $k=1$ is the qualitative statement of Baker's linear forms in
logarithms theorem.

\subsection{Proof of inequality~\eqref{eqn:maxBlogPvle0}}
\label{appendix3}\hfill\break
Let $B=(\b_{ij})$ and $P=(x_1,\ldots,x_N)$. Then
\begin{align*}
  \sum_{v\in M_K} \max B\log\|P\|_v
  &= \sum_{v\in M_K} \max_{1\le i\le N} 
         \left\{\sum_{j=1}^N \b_{ij}\log\|x_j\|_v\right\} \\
  &\ge \max_{1\le i\le N} \left\{\sum_{v\in M_K} 
      \left( \sum_{j=1}^N \b_{ij}\log\|x_j\|_v\right)\right\} \\
  &\ge \max_{1\le i\le N} \left\{\sum_{j=1}^N \b_{ij}
      \left( \sum_{v\in M_K} \log\|x_j\|_v\right)\right\} \\
  &=0,
\end{align*}
where the first inequality is due to the fact that a sum of maxs may
be strictly larger than the max of the sum, and where the last
equality follows from the product formula. (Note that all of the~$x_j$
are nonzero by assumption, since $P\in\GG_m^N(\Qbar)$.

\subsection{Description of $W_1\dotplus\cdots\dotplus W_r\dotplus Z$
as a kernel}
\label{appendix1}\hfill\break
The space $W_1\dotplus\cdots\dotplus W_r\dotplus Z$
is the kernel of the matrix
\[
  \prod_{|\l|<\rho} (A-\l)^N \cdot \prod_{|\l|=\rho} (A-\l)^\ell
  \in\Mat_N(\Qbar).
\]
This follows from the fact if $V\subset\Qbar^N$
is a Jordan subspace for~$A$ with eigenvalue~$\l$, then:
\begin{itemize}
\item
If $|\l|<\rho$, then $(A-\l)^N$ annihilates~$V$.
\item
If $|\l|=\rho$ and $\dim V\le\ell$, then $(A-\l)^\ell$
annihilates~$V$.
\item
If $|\l|=\rho$ and $\dim V=\ell$, so~$V$ is maximal, then
$(A-\l)^\ell$ annihilates~$W$, while~$(A-\l)^\ell$ acts as a nonzero
scalar on~$V/W$; cf.\ \eqref{eqn:WeqkerAlt1V}. 
\end{itemize}

\subsection{The smallest group containing a $\f$-orbit 
is $\f$-in\-var\-iant}\hfill\break
\label{appendix6}
Let $P\in\GG_m^N(\CC)$.  The claim is that if $G\subset\GG_m^N$ the
the smallest algebraic subgroup of~$\GG_m^N$ containing the
orbit~$\Orbit_\f(P)$, then~$\f(G)\subset G$. We remark that the proof
works more generally if~$\GG_m^N$ is replaced by any (commutative)
algebraic group~$A$ (over a field of characteristic~$0$). 
\par
So we let~$A$ be such a group, let~$\f:A\to A$ be an algebraic
homomorphism, and let~$\a\in A(\CC)$ be a point.  For each $n\ge1$,
let $H_n\subset A$ be the Zariski closure of the subgroup
of~$A$ generated by~$\f^n(\a)$,
\[
  H_n = \overline{\{ \f^n(\a)^k : k\in\ZZ\}}.
\]
Then~$H_n$ is a Zariski closed subset of~$A$. Further, the fact
that~$H_n$ is the closure of a subgroup implies that~$H_n$ is closed
under the group law, so~$H_n$ is an algebraic subgroup of~$A$.
\par
The group~$G$ contains~$\Orbit_\f(\a)$ by assumption, and its closed
and a group, so it contains all of the~$H_n$. We claim that~$G$ is the
smallest algebraic subgroup containing all of the~$H_n$. To see this,
suppose that~$G'$ is an algebraic subgroup and $G\supset H_n$ for all $n\ge0$.
Since~$\f^n(\a)\in H_n$, this implies in particular that $\f^n(\a)\in G'$,
and since this holds for all~$n\ge0$, we see that $\Orbit_\f(\a)\subset G'$.
But~$G$ is the smallest algebraic subgroup containing~$\Orbit_\f(\a)$,
so~$G\subset G'$. This proves the claim.
\par
We now consider a chain of containments:
\begin{align*}
  G
  \supset G \setminus H_0 
  &\supset \bigcup_{n\ge1} H_n
    &&\text{since $G$ contains $\displaystyle\bigcup_{n\ge0} H_n$,} \\
  &= \bigcup_{n\ge0} \f(H_n)
    &&\text{since $\f(H_n)=H_{n+1}$,} \\
  &= \f\left(\bigcup_{n\ge0} H_n\right).
\end{align*}
Hence
\[
  \f^{-1}(G) \supset \bigcup_{n\ge0} H_n.
\]
But~$\f^{-1}(G)$ is an algebraic subgroup of~$A$, while we showed
earlier that~$G$ is the smallest algebraic subgroup of~$A$ that
contains $\bigcup_{n\ge0} H_n$. Hence $G\subset\f^{-1}(G)$, which gives
the desired inclusion $\f(G)\subset G$.
\par
We remark that it is possible for~$\f(G)$ to be strictly contained
in~$G$.  For example, this happens if~$G$ has more than one connected
component and~$\f(G)$ is contained in the identity component of~$G$.

\subsection{Arithmetic degree independent of ambient space}
\label{appendix7}\hfill\break
We claimed that if~$i:\PP^k\to\PP^N$ with $\dim i(\PP^k)=k$, then
$h_{\PP^N}\circ i\asymp h_{\PP^k}$. Assuming this, consider any map
$\f:\PP^N\dashrightarrow\PP^N$ that descends to a
map~$\psi:\PP^k\dashrightarrow\PP^k$, i.e., so that
$i\circ\psi=\f\circ i$.  Then for any~$P\in\PP^k(\Qbar)$ we have
\[
  h_{\PP^N}\bigl(\f^n(i(P))\bigr)
  =   h_{\PP^N}\bigl(i(\psi^n(P))\bigr)
  \asymp h_{\PP^k}\bigl(\psi^n(P)\bigr).
\]
Taking~$n^{\text{th}}$-roots and the limsup as $n\to\infty$, we see that
$\a_\f\bigl(i(P)\bigr)=\a_\psi(P)$, i.e., we get the same value for
the arithmetic degree regardless of where we do the computation.
\par
The claim is easy, since~$i$ satisfies
$i^*\Ocal_{\PP^N}(1)=\Ocal_{\PP^k}(m)$ for some~$m\ge1$. (These are
line bundles, not orbits.)  It follows from standard properties of
heights that $h_{\PP^N}\circ i = m h_{\PP^k} + O(1)$, which is
stronger than what we claimed.

\subsection{A Jordan block condition that implies $\a_{\f}(P)<\d_{\f}$}
\label{appendix8}\hfill\break
We proved that if~$A$ is diagonalizable, then $\a_{\f}(P)<\d_{\f}$.  A
variation of the same argument can be used to prove the implication
\[
  \hhat_{\f}(P)=0\Longrightarrow\a_{\f}(P)<\d_{\f}
\]
under the weaker hypothesis
\begin{equation}
  \label{eqn:Jhyp}
  \tag{$*$}
  \left(
  \begin{tabular}{@{}l@{}}
       Every Jordan subspace for~$A$ whose eigenvalue~$\l$ \\
       satisfies $|\l|=\rho(A)$ is a maximal Jordan subspace\\
  \end{tabular}
  \right).
\end{equation}
Note that since the Jordan blocks of a diagonalizable matrix have
dimension~$1$, such matrices clearly satisfy~\eqref{eqn:Jhyp}.

\subsection{Verification of the implication in Remark \ref{remark:hhatasymph}}
\label{appendix4}\hfill\break
If $\hhat\asymp h$ and~$\hhat_\f(P)=0$, then
\begin{align*}
  h\bigl(\f^n(P)\bigr)
  &\asymp \hhat\bigl(\f^n(P)\bigr)\\
  &= \hhat_\f\bigl(\f^n(P)\bigr)+\hhat_{\f^{-1}}\bigl(\f^n(P)\bigr)\\
  &=\d_\f^n\hhat_\f(P) + \d_{\f^{-1}}^{-n}\hhat_{\f^{-1}}(P)\\
  &= \d_{\f^{-1}}^{-n}\hhat_{\f^{-1}}(P).
\end{align*}
Thus~$\bigl\{\f^n(P)\bigr\}$ is a set of bounded height, hence
finite, so~$P$ is periodic.

\end{document}